\documentclass[10pt,english,reqno]{amsart}

\usepackage{amssymb,latexsym,amsmath,epsfig,amsthm}
\usepackage{mathtools,mleftright}
\usepackage{booktabs,float,caption}
\usepackage{enumitem,cite}

\usepackage[
  pdfauthor={},
  pdfkeywords={},
  pdftitle={},
  pdfcreator={},
  pdfproducer={},
  linktocpage,colorlinks,bookmarksnumbered,linkcolor=blue,
  citecolor=red,urlcolor=red]{hyperref}

\theoremstyle{plain}
\newtheorem{theorem}{Theorem}
\newtheorem{lemma}[theorem]{Lemma}
\newtheorem{corollary}[theorem]{Corollary}

\theoremstyle{definition}
\newtheorem{example}[theorem]{Example}
\newtheorem{conjecture}[theorem]{Conjecture}
\newtheorem*{remark}{Remark}
\newtheorem*{question}{Question}

\numberwithin{theorem}{section}
\numberwithin{equation}{section}
\numberwithin{table}{section}

\newcommand{\NN}{\mathbb{N}}
\newcommand{\ZZ}{\mathbb{Z}}
\newcommand{\QQ}{\mathbb{Q}}
\newcommand{\RR}{\mathbb{R}}
\newcommand{\DD}{\mathbb{D}}

\newcommand{\CN}{\mathcal{C}}
\newcommand{\CP}{\mathcal{C}'}
\newcommand{\CS}{\mathcal{C}^{\sharp}}

\newcommand{\SF}{\mathbb{S}}
\newcommand{\SH}{\mathcal{S}}
\newcommand{\SD}{\mathfrak{S}'}
\newcommand{\SL}{\overline{\mathfrak{S}'}}
\newcommand{\SDG}{\mathfrak{S}}
\newcommand{\SLG}{\overline{\mathfrak{S}}}
\newcommand{\SG}{\mathbf{S}}

\newcommand{\RS}{\mathcal{R}}
\newcommand{\rr}{\mathbf{r}}

\newcommand{\GN}{\mathbf{G}}
\newcommand{\TN}{\mathbf{T}}
\newcommand{\HN}{\mathbf{H}}
\newcommand{\Ta}{\mathrm{Ta}}
\newcommand{\Tc}{\mathrm{Tc}}

\DeclareMathOperator{\denom}{denom}
\DeclareMathOperator{\pval}{v\mspace{-1.5mu}}
\DeclareMathOperator{\ord}{ord}
\DeclareMathOperator{\lcm}{lcm}
\DeclareMathOperator{\rad}{rad}

\newcommand{\Case}[1]{\mbox{\texttt{Case} #1.}}

\newcommand{\iffq}{\; \iff \;}

\newcommand{\impliesq}{\; \implies \;}
\newcommand{\impliest}{\quad \text{implies} \quad}
\newcommand{\andq}{\quad \text{and} \quad}
\newcommand{\withq}{\quad \text{with} \quad}

\newcommand{\set}[1]{\mleft\{#1\mright\}}
\newcommand{\intpart}[1]{\mleft[#1\mright]}
\newcommand{\fracpart}[1]{\mleft\{#1\mright\}}
\newcommand{\pmods}[1]{(\mathrm{mod}\;#1)}
\newcommand{\phrase}[1]{``#1"}
\newcommand{\fname}[1]{{\small \texttt{#1}}}
\newcommand{\pref}[1]{\texorpdfstring{\ref{#1}}{\ref*{#1}}}

\setlist[enumerate]{label=(\roman*),font=\rm,leftmargin=1.2cm,itemsep=1pt,
parsep=1pt,before={\parskip=1pt}}
\begin{document}

\title{On primary Carmichael numbers}
\author{Bernd C. Kellner}
\address{G\"ottingen, Germany}
\email{bk@bernoulli.org}


\begin{abstract}
The primary Carmichael numbers were recently introduced as a special subset of
the Carmichael numbers. A primary Carmichael number~$m$ has the unique property
that $s_p(m) = p$ holds for each prime factor $p$, where $s_p(m)$ is the sum of
the base-$p$ digits of $m$. The first such number is Ramanujan's famous taxicab
number $1729$. Due to Chernick, all Carmichael numbers with three factors can be
constructed by certain squarefree polynomials $U_3(t) \in \mathbb{Z}[t]$, the
simplest one being \mbox{$U_3(t) = (6t+1)(12t+1)(18t+1)$}.
We show that the values of any $U_3(t)$ obey a special decomposition for all
$t \geq 2$ and besides certain exceptions also in the case $t=1$. These cases
further imply that if all three factors of $U_3(t)$ are simultaneously odd
primes, then $U_3(t)$ is not only a Carmichael number, but also a primary
Carmichael number. Together with the exceptional cases, all Carmichael numbers
with three factors have at least the property that $s_p(m) = p$ holds for the
greatest prime factor $p$ of $m$. Subsequently, we show some connections to
taxicab and polygonal numbers, involving the number $1729$ as an example again.
\end{abstract}

\maketitle


\section{Introduction}

By Fermat's little theorem the congruence
\[
  a^{m-1} \equiv 1 \pmod{m}
\]
holds for all integers $a$ coprime to $m$, if $m$ is a prime.
Moreover, this congruence also holds for positive composite integers $m$,
which are called \emph{Carmichael numbers} and obey the following criterion.
Let $p$ always denote a prime.

\begin{theorem}[Korselt's criterion \cite{Korselt:1899} (1899)]
A positive composite integer $m$ is a Carmichael number if and only if $m$ is
squarefree and
\[
  p \mid m \impliesq p - 1 \mid m - 1.
\]
\end{theorem}

Subsequently, Carmichael independently derived further properties of these
numbers and computed first examples of them.

\begin{theorem}[Carmichael \cite{Carmichael:1910,Carmichael:1912} (1910,1912)]
If $m$ is a Carmichael number, then $m$ is a positive odd and squarefree
integer having at least three prime factors.
Moreover, if $p$ and $q$ are prime divisors of $m$, then
\[
  p - 1 \mid m - 1, \quad p - 1 \mid \frac{m}{p} - 1, \andq p \nmid q - 1.
\]
\end{theorem}

Denote the set of Carmichael numbers by
\begin{align*}
  \CN = \{ &561, 1105, 1729, 2465, 2821, 6601, 8911, 10\,585, 15\,841, 29\,341, \\
  &41\,041, 46\,657, 52\,633, 62\,745, 63\,973, 75\,361, 101\,101, \dotsc \}.
\end{align*}

Following~\cite{Kellner&Sondow:2021}, the Carmichael numbers can be also
characterized in a quite different and surprising way.
Let $s_p(m)$ be the sum of the base-$p$ digits of~$m$.

\begin{theorem}[Kellner and Sondow \cite{Kellner&Sondow:2021}] \label{thm:CN-prop}
An integer $m > 1$ is a Carmichael number if and only if $m$ is squarefree
and each of its prime divisors $p$ satisfies both
\[
  s_p(m) \geq p \andq s_p(m) \equiv 1 \pmod{p-1}.
\]
Moreover, $m$ is odd and has at least three prime factors,
each prime factor $p$ obeying the sharp bound
\[
  p \leq \alpha \, \sqrt{m} \withq \alpha = \sqrt{17/33} = 0.7177\dotsc.
\]
\end{theorem}

Define the set of \emph{primary Carmichael numbers} by
\[
  \CP := \set{m \in \SF \,:\, p \mid m \impliesq s_p(m) = p},
\]
where $\SF = \set{2,3,5,6,7,10,\dotsc}$ is the set of squarefree integers
$m > 1$. The first elements are given by
\begin{align*}
  \CP = \{ &1729, 2821, 29\,341, 46\,657, 252\,601, 294\,409, 399\,001, 488\,881, \\
  &512\,461, 1\,152\,271, 1\,193\,221, 1\,857\,241, 3\,828\,001, 4\,335\,241, \dotsc \}.
\end{align*}

The set $\CP$ (meaning \phrase{$\CN$ prime}) of primary Carmichael numbers,
which was introduced in \cite{Kellner&Sondow:2021}, is indeed a subset of the
Carmichael numbers.

\begin{theorem}[Kellner and Sondow \cite{Kellner&Sondow:2021}] \label{thm:CP-prop}
We have $\CP \subset \CN$. If $m \in \CP$, then
each prime factor $p$ of $m$ obeys the sharp bound
\[
  p \leq \alpha \, \sqrt{m} \withq \alpha = \sqrt{66\,337/132\,673} = 0.7071\dotsc.
\]
\end{theorem}

We further define for a given set $\SG \subseteq \CN$ the subsets
$\SG_n \subseteq \SG$, where each element of $\SG_n$ has exactly $n$ prime factors.
Let $\mathrm{S}(x)$ and $\mathrm{S}_n(x)$ count the number of elements of $\SG$
and $\SG_n$ less than $x$, respectively. We call a squarefree number $m$ with
exactly~$n$ prime factors briefly an \emph{$n$-factor} number.

The first element of $\CP_n$ for $n=3,4,5$ is given by
\begin{align*}
  1729 &= 7 \cdot 13 \cdot 19, \\
  10\,606\,681 &= 31 \cdot 43 \cdot 73 \cdot 109, \\
  4\,872\,420\,815\,346\,001 &= 211 \cdot 239 \cdot 379 \cdot 10\,711 \cdot 23\,801,
\end{align*}
respectively.

In 1939 Chernick~\cite{Chernick:1939} introduced certain squarefree polynomials
\[
  U_n(t) \in \ZZ[t] \text{ of degree } n \geq 3
\]
to construct Carmichael numbers, where $t \geq 0$ is an integer. More precisely,
he showed that $U_n(t)$ represents a Carmichael number for $t \geq 0$,
whenever all $n$ linear factors of $U_n(t)$ are simultaneously odd primes.
The simplest one of these polynomials is
\begin{equation} \label{eq:U3-simple}
  U_3(t) = (6 \, t + 1)(12 \, t + 1)(18 \, t + 1),
\end{equation}
which produces the $3$-factor Carmichael numbers
\begin{alignat*}{3}
  1729 &= 7 \cdot 13 \cdot 19 &\quad& (t=1), \\
  294\,409 &= 37 \cdot 73 \cdot 109 && (t=6), \\
  56\,052\,361 &= 211 \cdot 421 \cdot 631 && (t=35),
\end{alignat*}
being the first three examples.

At first glance, one observes that the third-smallest Carmichael number $1729$,
which is also known as Ramanujan's famous \emph{taxicab number}
(being the smallest number that is a sum of two positive cubes in two ways,
see Silverman \cite{Silverman:1993}), namely,
\begin{equation} \label{eq:taxicab}
  1729 = 1^3 + 12^3 = 9^3 + 10^3,
\end{equation}
is additionally the smallest primary Carmichael number.
Surprisingly, a closer look reveals that the other two numbers $294\,409$
and $56\,052\,361$ are also primary Carmichael numbers.
Is this pure coincidence or a hidden phenomenon?

The purpose of this paper is to show that any $U_3(t)$ has the property that
all values of $U_3(t)$ for $t \geq 2$, and apart from certain exceptions also
in the case $t=1$, lie in a certain set $\SD$ (as introduced in
Section~\ref{sec:decomp}) that generalizes the set~$\CP$.

As a main result of Section~\ref{sec:forms}, it further turns out that any
given $U_3(t)$ has the following important property: if both $U_3(t) \in \SD$
and all three linear factors of $U_3(t)$ are odd primes for a fixed $t \geq 0$,
then $U_3(t)$ represents not only a Carmichael number, but also a primary
Carmichael number.

Thus, almost all $3$-factor Carmichael numbers, which were computed by
Chernick's method so far, lie in $\CP_3$. The restriction \phrase{almost}
refers to the exceptions in the cases $t=0$ and $t=1$.

As a striking example, in 1980 Wagstaff~\cite{Wagstaff:1980} already computed
a very huge \mbox{$3$-factor} Carmichael number with $321$ decimal digits by
using $U_3(t)$ as defined by~\eqref{eq:U3-simple}, where $t$ is a $106$-digit
number. This number now awakes from a deep sleep as a primary Carmichael
number!

In 2002 Dubner~\cite{Dubner:2002} also used this $U_3(t)$ to compute the
corresponding $3$-factor Carmichael numbers up to $10^{42}$,
which are all primary.

By this means, one can even find a special $\tilde{U}_3(t)$ very quickly such
that for $t=1$ the value $M = \tilde{U}_3(1) \in \CP_3$ yields the large example
\begin{align*}
  M &= 37\,717\,531\,166\,520\,286\,365\,396\,946\,681 \\
    &= 1\,570\,642\,921 \cdot 3\,094\,633\,081 \cdot 7\,759\,909\,081,
\end{align*}
satisfying in fact the remarkable property
\[
  s_p(M) = p
\]
for each prime factor $p$ of $M$. The reader is invited to check this property
above. See Table~\ref{tbl:Ur-large} in Section~\ref{sec:forms} for the
construction.

In 1904 Dickson~\cite{Dickson:1904} stated the conjecture that a set of linear
functions $f_\nu(t) = a_\nu t + b_\nu \in \ZZ[t]$, under certain conditions,
might be simultaneously prime for infinitely many integral values of $t$.

Hence, Dickson's conjecture, as already noted by Chernick, implies that any
$U_3(t)$ produces infinitely many Carmichael numbers, and so the set $\CN$
should be infinite. This statement now transfers to the set $\CP$ of primary
Carmichael numbers.

While the question, whether there exist infinitely many Carmichael numbers,
was positively answered by Alford, Granville, and Pomerance \cite{AGP:1994}
in 1994, the related question for the primary Carmichael numbers and their
distribution is still open.

Unfortunately, several computations suggest that the properties of $U_3(t)$
as described above do not hold for $U_n(t)$ with $n \geq 4$. One may speculate
whether this causes the high proportion of primary Carmichael numbers with
exactly three prime factors among all primary Carmichael numbers, see
Table~\ref{tbl:distrib}. However, we raise an explicit conjecture on related
properties of $U_4(t)$ in Section~\ref{sec:forms}.

Going into more detail, Table~\ref{tbl:distrib} shows the distributions of
$C(x)$, $C'(x)$, and their subsets up to $10^{18}$. On the one hand, one observes
that in this range about $97\%$ of the primary Carmichael numbers have exactly
three factors, the remaining $3\%$ have four and five factors. On the other hand,
the ratio $C'_3(x)/C_3(x)$ is steadily increasing for $x$ in the range up to
$10^{18}$, implying that about $87\%$ of the $3$-factor Carmichael numbers are
primary in that range.

\begin{table}[H]
\begin{center} \small
\begin{tabular}{lrrrrrccc}
  \toprule
  \multicolumn{1}{c}{$x$} & \multicolumn{1}{c}{$C(x)$} &
  \multicolumn{1}{c}{$C_3(x)$} & \multicolumn{1}{c}{$C'(x)$} &
  \multicolumn{1}{c}{$C'_3(x)$} & \multicolumn{1}{c}{$C'_4(x)$} &
  \multicolumn{1}{c}{$C'_5(x)$} & \multicolumn{1}{c}{$C'_3/C'(x)$} &
  \multicolumn{1}{c}{$C'_3/C_3(x)$} \\
  \midrule
  $10^{3}$ & $1$ & $1$ & & & & & --- & --- \\
  $10^{4}$ & $7$ & $7$ & $2$ & $2$ & & & $1.000$ & $0.286$ \\
  $10^{5}$ & $16$ & $12$ & $4$ & $4$ & & & $1.000$ & $0.333$ \\
  $10^{6}$ & $43$ & $23$ & $9$ & $9$ & & & $1.000$ & $0.391$ \\
  $10^{7}$ & $105$ & $47$ & $19$ & $19$ & & & $1.000$ & $0.404$ \\
  $10^{8}$ & $255$ & $84$ & $51$ & $48$ & $3$ & & $0.941$ & $0.571$ \\
  $10^{9}$ & $646$ & $172$ & $107$ & $104$ & $3$ & & $0.972$ & $0.605$ \\
  $10^{10}$ & $1547$ & $335$ & $219$ & $214$ & $5$ & & $0.977$ & $0.639$ \\
  $10^{11}$ & $3605$ & $590$ & $417$ & $409$ & $8$ & & $0.981$ & $0.693$ \\
  $10^{12}$ & $8241$ & $1000$ & $757$ & $741$ & $16$ & & $0.979$ & $0.741$ \\
  $10^{13}$ & $19\,279$ & $1858$ & $1470$ & $1433$ & $37$ & & $0.975$ & $0.771$ \\
  $10^{14}$ & $44\,706$ & $3284$ & $2666$ & $2599$ & $67$ & & $0.975$ & $0.791$ \\
  $10^{15}$ & $105\,212$ & $6083$ & $5040$ & $4896$ & $144$ & & $0.971$ & $0.805$ \\
  $10^{16}$ & $246\,683$ & $10\,816$ & $9280$ & $8996$ & $282$ & $2$ & $0.969$ & $0.832$ \\
  $10^{17}$ & $585\,355$ & $19\,539$ & $17\,210$ & $16\,694$ & $514$ & $2$ & $0.970$ & $0.854$ \\
  $10^{18}$ & $1\,401\,644$ & $35\,586$ & $32\,039$ & $31\,103$ & $933$ & $3$ & $0.971$ & $0.874$ \\
  \bottomrule
\end{tabular}

\caption{\parbox[t]{20.5em}{Distributions of $C(x)$, $C'(x)$, and their subsets.\\
The ratios are rounded to three decimal places.}}
\label{tbl:distrib}
\end{center}
\end{table}

Computed Carmichael numbers and tables up to $10^{18}$ in this paper were taken
from Pinch's tables in \cite{Pinch:2007,Pinch:2008}, while the numbers up
to~$10^9$, in particular for $\CP$, were rechecked by our computations. Further
tables are given by Granville and Pomerance in \cite{Granville&Pomerance:2002},
which also rely mainly on Pinch's computations.
The used raw data files of \cite{Pinch:2008} are named \fname{carmichael-16.gz},
\fname{carmichael17.gz}, \fname{carmichael18.gz}, and \fname{car3-18.gz}.

Interestingly, the progress about the (primary) Carmichael numbers, as partially
described above, were originally initiated by a completely different context.
For the sake of completeness, we give here a short survey of some results of
\cite{Kellner:2017,Kellner&Sondow:2017,Kellner&Sondow:2018,Kellner&Sondow:2021}.

As usual, denote the Bernoulli polynomials and numbers by $B_n(x)$ and
$B_n = B_n(0)$, respectively. The polynomials $B_n(x)$ are defined by the
series (cf.~\cite[Sec.~9.1, pp.~3--4]{Cohen:2007})
\[
  \frac{z e^{xz}}{e^z - 1} = \sum_{n \geq 0} B_n(x) \frac{z^n}{n!}
  \quad (|z| < 2\pi).
\]
Define for $n \geq 1$ the denominators $\DD_n := \denom(B_n(x) - B_n)$
of the Bernoulli polynomials, which have no constant term,
\[
  B_n(x) - B_n = \sum_{k=0}^{n-1} \binom{n}{k} B_k \, x^{n-k}.
\]
These denominators are given by the notable formula
\[
  \DD_n = \prod_{s_p(n) \, \geq \, p} p
\]
and obey several divisibility properties. We have, for example,
\begin{align*}
  \rad(n+1) \mid \DD_n, &\quad \text{if $n+1$ is composite}, \\
  \DD_n = \lcm(\DD_{n+1}, \rad(n+1)), &\quad \text{if $n \geq 3$ is odd},
\end{align*}
where $\rad(n) := \prod_{p \, \mid \, n} p$. It further turns out that
all Carmichael numbers satisfy the divisibility relation
\[
  m \in \CN \impliesq m \mid \DD_m,
\]
which explains the unexpected link between Carmichael numbers and the function
$s_p(\cdot)$.

The rest of the paper is organized as follows. The main results, theorems, and
conjectures are presented in Sections~\ref{sec:decomp}~--~\ref{sec:compl}
after introducing necessary definitions and complementary results.
Subsequently, Sections~\ref{sec:proofs-1}~--~\ref{sec:proofs-3} contain the
proofs of the theorems, ordered by their dependencies.
Section~\ref{sec:taxicab} shows some connections to the taxicab numbers.
Finally, in Section~\ref{sec:poly} we give applications to the polygonal numbers.


\section{Decompositions}
\label{sec:decomp}

Let $\NN$ be the set of positive integers. The sum-of-digits function
$s_p(\cdot)$ is actually defined for any integer base $g \geq 2$ in place of
a prime~$p$. To avoid ambiguity, we define $s_1(m) := 0$ for $m \geq 0$.
For integers $g \geq 2$ and $m \geq 1$ define
\[
  \ord_g(m) := \max \set{n \geq 0 : g^n \mid m}.
\]

We say that a positive integer $m$ has an \emph{$s$-decomposition},
if there exists a decomposition in $n$ proper factors $g_\nu$ with
exponents~$e_\nu \geq 1$, the factors $g_\nu$ being strictly increasing but
not necessarily coprime, such that
\begin{equation} \label{eq:s-decomp}
  m = \prod_{\nu=1}^{n} g_\nu^{e_\nu},
\end{equation}
where each factor $g_\nu$ satisfies the \emph{sum-of-digits condition}
\begin{equation} \label{eq:s-condg}
  s_{g_\nu}(m) \geq g_\nu.
\end{equation}
Similarly, we say that \eqref{eq:s-decomp} represents a \emph{strict}
$s$-decomposition, if each factor $g_\nu$ satisfies the strict sum-of-digits
condition
\begin{equation} \label{eq:s-conds}
  s_{g_\nu}(m) = g_\nu.
\end{equation}

Accordingly, we define the sets
\begin{align*}
  \SDG &:= \set{m \in \NN : m \text{ has an $s$-decomposition}},\\
  \SD  &:= \set{m \in \NN : m \text{ has a strict $s$-decomposition}}.
\end{align*}
One computes that
\begin{align*}
  \SDG = \{ &24, 45, 48, 72, 96, 120, 144, 189, 192, 216, 224, 225, 231, 240, \\
            &280, 288, 315, 320, 325, 336, 352, 360, 378, 384, 405, 432, \dotsc \}, \\
  \SD  = \{ &45, 96, 225, 325, 405, 576, 637, 640, 891, 1225, 1377, 1408, 1536, \\
            &1701, 1729, 2025, 2541, 2821, 3321, 3751, 3825, 4225, 4608, \dotsc \}.
\end{align*}

Clearly, we have $\SD \subset \SDG$. Some examples of $s$-decompositions are
\[
  45 = 3^2 \cdot 5, \, 576 = 2^4 \cdot 6^2, \, 1729 = 7 \cdot 13 \cdot 19, \,
  2025 = 5^2 \cdot 9^2.
\]
Note that an $s$-decomposition of a number $m \in \SDG$ does not have to be unique.
Such an example of different $s$-decompositions is given by
\[
  240 = 2^4 \cdot 3 \cdot 5 = 2^2 \cdot 3 \cdot 4 \cdot 5
    = 2 \cdot 3 \cdot 5 \cdot 8 = 3 \cdot 4^2 \cdot 5,
\]
showing all possible variants.

While the definition of the set~$\SD$ widely extends the definition of the
set~$\CP$, the set~$\SDG$ widely extends the set
\[
  \SH := \set{m \in \SF \,:\, p \mid m \impliesq s_p(m) \geq p}
\]
where
\begin{align*}
  \SH = \{ &231, 561, 1001, 1045, 1105, 1122, 1155, 1729, 2002, 2093, \\
           &2145, 2465, 2821, 3003, 3315, 3458, 3553, 3570, 3655, \dotsc \}.
\end{align*}

As introduced and shown in \cite{Kellner&Sondow:2021}, the set $\SH$ has the
property that $\CN \subset \SH$. Moreover, each number $m \in \SH$ has at least
three prime factors.

The next two theorems summarize the properties of $\SDG$ and $\SD$,
which also show some connections with the Carmichael numbers.

\begin{theorem} \label{thm:s-decomp}
An $s$-decomposition of $m \in \SDG$ has the following properties:
\begin{enumerate}
\item The $s$-decomposition of $m$ has at least two factors, while $m$ has at
      least two prime divisors.
\item If $m = g_1^{e_1} \cdot g_2^{e_2}$, then $e_1 + e_2 \geq 3$.
\item If $m = g_1 \cdot g_2 \cdot g_3$ where all $g_\nu$ are odd primes,
      then its $s$-decomposition is unique. In particular, if $m \in \SD$,
      then $m \in \CP_3$.
\item If $m = g_1 \dotsm g_n$ with $n \geq 3$, where all $g_\nu$ are odd primes,
      then $m \in \SH$. Moreover, if $g_1 \dotsm g_n \in \SD$, then $m \in \CP_n$.
\item If $m = g_1^{e_1} \dotsm g_n^{e_n}$ with $n \geq 2$, then each factor
      $g_\nu$ satisfies the inequalities
      $1 < g_\nu < m^{1/(\ord_{g_\nu}(m)+1)} \leq m^{1/(e_\nu+1)} \leq \sqrt{m}$.
\end{enumerate}
\end{theorem}

\begin{theorem} \label{thm:SD-prop}
The sets $\SDG$ and $\SD$ have the following properties:
\begin{enumerate}
\item $\CN \subset \SDG$.
\item $\CP \subset \SD \,\cap\, \CN$.
\item $\CP_3 = \SD \,\cap\, \CN_3$.
\end{enumerate}
\end{theorem}

We further define the generalized sets of $\SDG$ and $\SD$ by
\begin{align*}
  \SLG &:= \set{m \in \NN : \, \text{there exists $g \mid m$ with $s_g(m) \geq g$}}, \\
  \SL  &:= \set{m \in \NN : \, \text{there exists $g \mid m$ with $s_g(m) = g$}}.
\end{align*}

The sets $\SLG$ and $\SL$ satisfy the conditions~\eqref{eq:s-condg}
and~\eqref{eq:s-conds} for at least one proper divisor of each
of their elements, respectively. One computes that
\begin{align*}
  \SLG = \{ &6, 10, 12, 14, 15, 18, 20, 21, 22, 24, 26, 28, 30, 33, 34, 36, 38, 39,\\
            &40, 42, 44, 45, 46, 48, 50, 51, 52, 54, 56, 57, 58, 60, 62, 63, \dotsc \},\\
  \SL  = \{ &6, 10, 12, 15, 18, 20, 21, 24, 28, 33, 34, 36, 39, 40, 45, 48, 52,\\
            &57, 63, 65, 66, 68, 72, 76, 80, 85, 87, 88, 91, 93, 96, 99, 100, \dotsc \}.
\end{align*}

By the definitions and the computed examples we have the relations
\begin{equation}\label{eq:SD-rel}
  \SD \subset \SDG \subset \SLG \andq \SD \subset \SL \subset \SLG.
\end{equation}

The following two theorems show weaker and different properties of~$\SLG$
and~$\SL$ compared to Theorems~\ref{thm:s-decomp} and~\ref{thm:SD-prop}.

\begin{theorem} \label{thm:SL-elem}
A number $m \in \SLG$ and a divisor $g \mid m$ with $s_g(m) \geq g$
have the following properties:
\begin{enumerate}
\item $m$ has at least two prime divisors.
\item If $m \in \CN_3$, then $g$ is an odd prime.
\item $g$ obeys the inequalities $1 < g < m^{1/(\ord_{g}(m)+1)} \leq \sqrt{m}$.
\end{enumerate}
\end{theorem}

\begin{theorem} \label{thm:SL-prop}
The set $\SL \setminus \SD$ has the following properties:
\begin{enumerate}
\item $\CN \setminus \CP \not\subset \SL \setminus \SD$.
\item $\CN_3 \setminus \CP_3 \subset \SL \setminus \SD$.
\end{enumerate}
\end{theorem}

\begin{remark}
Theorems~\ref{thm:SL-elem} (ii) and~\ref{thm:SL-prop} (ii), and the
properties of the set $\CP_3$ imply that all $3$-factor Carmichael numbers have
the following property: every number $m \in \CN_3$ satisfies the strict
sum-of-digits condition~\eqref{eq:s-conds} for at least one prime factor of $m$.
This will be stated later more precisely; see Theorems~\ref{thm:main},
\ref{thm:main2}, and~\ref{thm:except}.
\end{remark}

If one could show the open question, whether the set $\CP$ is infinite, then
Theorem~\ref{thm:SD-prop} would imply that $\SD$ is also infinite. Fortunately,
the infinitude of $\SD$ can be shown independently of the set $\CP$.

\begin{theorem} \label{thm:SD-inf}
The set $\SD$ is infinite.
\end{theorem}

The relations in \eqref{eq:SD-rel} immediately imply the following corollary.

\begin{corollary}
The sets $\SDG$, $\SLG$, and $\SL$ are infinite.
\end{corollary}

Finally, we define the subsets $\SDG_*$ and $\SD_*$ of the sets $\SDG$ and $\SD$,
respectively. Each element $m \in \SDG_*$ (respectively, $m \in \SD_*$) has the
property that the prime factorization of $m$ equals a (strict) $s$-decomposition.
The definitions are given as
\begin{align*}
  \SDG_* &:= \set{m \in \NN_{\geq 2} \,:\, p \mid m \impliesq s_p(m) \geq p}, \\
  \SD_*  &:= \set{m \in \NN_{\geq 2} \,:\, p \mid m \impliesq s_p(m) = p}.
\end{align*}
By Theorem~\ref{thm:CN-prop} and the definition of the set~$\CP$, we have the
relations
\[
  \CN \subset \SDG_* \subset \SDG \andq \CP \subset \SD_* \subset \SD.
\]

While for a given number~$m$ the determination of its $s$-decomposition may be
difficult due to searching for suitable factors (actually, this problem can be
translated into a system of linear equations), the sets $\SDG_*$ and $\SD_*$
can be computed quite easily by checking only prime factorizations.
The first numbers that do not have a \emph{trivial} (strict) $s$-decomposition
are given as follows.
\begin{align*}
  \SDG \setminus \SDG_* &= \set{280, 378, 640, 1134, 1280, 1408, 1430, 2464, 2520, 2816, \dotsc},\\
  \SD \setminus \SD_*   &= \set{96, 225, 576, 640, 1225, 1377, 1408, 1536, 1701, 2025, \dotsc}.
\end{align*}

Let $S(x)$ count the number of elements of $\SDG$ less than $x$; analogously
define this notation for related sets of $\SDG$. Table~\ref{tbl:distrib2} shows
their distributions compared to $C'(x)$ and $C(x)$.

\begin{table}[H]
\begin{center} \small
\begin{tabular}{l*{8}{r}}
  \toprule
  \multicolumn{1}{c}{$x$} & \multicolumn{1}{c}{$C'(x)$} & \multicolumn{1}{c}{$C(x)$} &
  \multicolumn{1}{c}{$S_*'(x)$} & \multicolumn{1}{c}{$S_*(x)$} &
  \multicolumn{1}{c}{$S'(x)$} & \multicolumn{1}{c}{$S(x)$} &
  \multicolumn{1}{c}{$\overline{S'}(x)$} & \multicolumn{1}{c}{$\overline{S}(x)$}\\
  \midrule
  $10^1$ & & & & & & & $1$ & $1$\\
  $10^2$ & & & $1$ & $5$ & $2$ & $5$ & $32$ & $60$ \\
  $10^3$ & & $1$ & $5$ & $53$ & $9$ & $56$ & $220$ & $742$ \\
  $10^4$ & $2$ & $7$ & $13$ & $477$ & $34$ & $532$ & $1401$ & $8050$ \\
  $10^5$ & $4$ & $16$ & $32$ & $4147$ & $100$ & $4837$ & $8388$ & $84\,057$ \\
  $10^6$ & $9$ & $43$ & $62$ & $35\,827$ & $254$ & 43\,981 & $51\,333$ & $864\,438$ \\
  \bottomrule
\end{tabular}

\caption{\parbox[t]{15.3em}{Distributions of $C'(x)$, $C(x)$, $S_*'(x)$, $S_*(x)$,
$S'(x)$, $S(x)$, $\overline{S'}(x)$, and $\overline{S}(x)$.}}
\label{tbl:distrib2}
\end{center}
\end{table}

At first glance, a lower bound for the growth of $S'(x)$ is given by $O(x^{1/3})$,
which will be implied by Theorem~\ref{thm:main} later. We show this lower bound
with explicit and simple constants.

\begin{theorem} \label{thm:SD-bound}
There is the estimate
\[
  S'(x) > \frac{1}{11} \, x^{1/3} - \frac{1}{3} \quad (x \geq 1).
\]
\end{theorem}


\section{Exceptional Carmichael Numbers}
\label{sec:except}

We introduce the set of \emph{exceptional Carmichael numbers} by
\[
  \CS := \set{m \in \CN \,:\, p \mid m \impliesq s_p(m) \neq p}.
\]
By definition we have
\[
   \CS \subseteq \CN \setminus \CP \andq
   \CS_n \subseteq \CN_n \setminus \CP_n \quad (n \geq 3).
\]
The first numbers in $\CS$ are
\begin{align*}
  173\,085\,121 &= 11 \cdot 31 \cdot 53 \cdot 61 \cdot 157, \\
  321\,197\,185 &= 5 \cdot 19 \cdot 23 \cdot 29 \cdot 37 \cdot 137, \\
  455\,106\,601 &= 19 \cdot 41 \cdot 53 \cdot 73 \cdot 151.
\end{align*}

In view of Theorem~\ref{thm:SL-prop}, the special properties of
the $3$-factor Carmichael numbers can be now restated as follows.

\begin{theorem} \label{thm:CS-3}
We have $\CS_3 = \emptyset$.
\end{theorem}

In the case of the $4$-factor Carmichael numbers, it seems that such
\emph{exceptions} occur very rarely.
Indeed, the set $\CS_4$ contains only four numbers below $10^{18}$:
\begin{align*}
  954\,732\,853 &= 103 \cdot 109 \cdot 277 \cdot 307, \\
  54\,652\,352\,931\,793 &= 1013 \cdot 2377 \cdot 2729 \cdot 8317, \\
  2\,948\,205\,156\,573\,601 &= 2539 \cdot 8101 \cdot 11\,551 \cdot 12\,409, \\
  456\,691\,406\,989\,839\,841 &= 8737 \cdot 31\,981 \cdot 38\,377 \cdot 42\,589.
\end{align*}

As a consequence of Theorem~\ref{thm:CN-prop}, each prime factor $p$ of $m \in \CS$
must satisfy both conditions $s_p(m) \geq 2p-1$ and $s_p(m) \equiv 1 \pmod{p-1}$.
Actually, one verifies that the first four numbers $m \in \CS_4$, as listed above,
even satisfy the condition
\[
  s_p(m) = 2p-1
\]
for each prime factor $p$ of $m$.

The $4$-factor Carmichael numbers seem to also play a particular role like the
$3$-factor Carmichael numbers. This will be discussed in the next section.
Tables~\ref{tbl:distrib3} and~\ref{tbl:distrib4} illustrate the distributions of
the sets~$\CS_n$ and~$\CN_n$, respectively. One also finds Table~\ref{tbl:distrib4}
in~\cite{Granville&Pomerance:2002}, but with values given up to $10^{16}$.

\begin{table}[H]
\begin{center} \small
\begin{tabular}{lrc*{7}{r@{\hspace*{1em}}}}
  \toprule
  \multicolumn{1}{c}{$x$} &
  \multicolumn{1}{c}{$C^{\sharp}(x)$} &
  \multicolumn{1}{c}{\hspace*{-0.8em}$C^{\sharp}_4(x)$} &
  \multicolumn{1}{c}{\hspace*{-0.8em}$C^{\sharp}_5(x)$} &
  \multicolumn{2}{c}{} & \multicolumn{1}{c}{$\dotsc$} & \multicolumn{2}{c}{} &
  \multicolumn{1}{r}{\hspace*{-2em}$C^{\sharp}_{11}(x)$}\\
  \midrule
  $10^{9}$  & $11$ & $1$ & $7$ & $3$ \\
  $10^{10}$ & $48$ & $1$ & $19$ & $27$ & $1$ \\
  $10^{11}$ & $169$ & $1$ & $49$ & $94$ & $25$ \\
  $10^{12}$ & $590$ & $1$ & $104$ & $346$ & $135$ & $4$ \\
  $10^{13}$ & $1780$ & $1$ & $194$ & $899$ & $622$ & $63$ & $1$ \\
  $10^{14}$ & $5456$ & $2$ & $397$ & $2326$ & $2252$ & $456$ & $23$ \\
  $10^{15}$ & $16\,245$ & $2$ & $692$ & $5482$ & $7504$ & $2420$ & $145$ \\
  $10^{16}$ & $47\,171$ & $3$ & $1227$ & $12\,149$ & $22\,287$ & $10\,293$ & $1189$ & $23$ \\
  $10^{17}$ & $136\,704$ & $3$ & $2205$ & $26\,464$ & $61\,640$ & $38\,886$ & $7187$ & $318$ & $1$ \\
  $10^{18}$ & $386\,066$ & $4$ & $3713$ & $54\,128$ & $158\,276$ & $131\,641$ & $35\,472$ & $2785$ & $47$ \\
  \bottomrule
\end{tabular}

\caption{Distributions of $C^{\sharp}(x)$ and $C^{\sharp}_4(x),
\dots, C^{\sharp}_{11}(x)$.}
\label{tbl:distrib3}
\end{center}
\end{table}

\begin{table}[H]
\begin{center} \scriptsize
\begin{tabular}{l*{10}{r@{\hspace*{0.6em}}}}
  \toprule
  \multicolumn{1}{c}{$x$} &
  \multicolumn{1}{c}{$C(x)$} & \multicolumn{1}{c}{$C_3(x)$} &
  \multicolumn{1}{c}{$C_4(x)$} & \multicolumn{1}{c}{$C_5(x)$} &
  \multicolumn{5}{c}{$\dotsc$} & \multicolumn{1}{r}{\hspace*{-2.3em}$C_{11}(x)$} \\
  \midrule
  $10^{3}$  & $1$ & $1$ \\
  $10^{4}$  & $7$ & $7$ \\
  $10^{5}$  & $16$ & $12$ & $4$ \\
  $10^{6}$  & $43$ & $23$ & $19$ & $1$ \\
  $10^{7}$  & $105$ & $47$ & $55$ & $3$ \\
  $10^{8}$  & $255$ & $84$ & $144$ & $27$ \\
  $10^{9}$  & $646$ & $172$ & $314$ & $146$ & $14$ \\
  $10^{10}$ & $1547$ & $335$ & $619$ & $492$ & $99$ & $2$ \\
  $10^{11}$ & $3605$ & $590$ & $1179$ & $1336$ & $459$ & $41$ \\
  $10^{12}$ & $8241$ & $1000$ & $2102$ & $3156$ & $1714$ & $262$ & $7$ \\
  $10^{13}$ & $19\,279$ & $1858$ & $3639$ & $7082$ & $5270$ & $1340$ & $89$ & $1$ \\
  $10^{14}$ & $44\,706$ & $3284$ & $6042$ & $14\,938$ & $14\,401$ & $5359$ & $655$ & $27$ \\
  $10^{15}$ & $105\,212$ & $6083$ & $9938$ & $29\,282$ & $36\,907$ & $19\,210$ & $3622$ & $170$ \\
  $10^{16}$ & $246\,683$ & $10\,816$ & $16\,202$ & $55\,012$ & $86\,696$ & $60\,150$ & $16\,348$ & $1436$ & $23$ \\
  $10^{17}$ & $585\,355$ & $19\,539$ & $25\,758$ & $100\,707$ & $194\,306$ & $172\,234$ & $63\,635$ & $8835$ & $340$ & $1$ \\
  $10^{18}$ & $1\,401\,644$ & $35\,586$ & $40\,685$ & $178\,063$ & $414\,660$ & $460\,553$ & $223\,997$ & $44\,993$ & $3058$ & $49$ \\
  \bottomrule
\end{tabular}

\caption{Distributions of $C(x)$ and $C_3(x), \dots, C_{11}(x)$.}
\label{tbl:distrib4}
\end{center}
\end{table}


\section{Universal Forms}
\label{sec:forms}

Chernick \cite{Chernick:1939} introduced so-called \emph{universal forms},
which are squarefree polynomials in $\ZZ[t]$, by
\begin{equation} \label{eq:Un-def}
  U_n(t) := \prod_{\nu=1}^{n} \, (a_\nu \, t + b_\nu) \quad (n \geq 3)
\end{equation}
with coefficients $a_\nu, b_\nu \in \NN$ satisfying
\begin{equation} \label{eq:Un-prop}
  U_n(t) \equiv 1 \pmod{a_\nu \, t + b_\nu - 1} \quad (1 \leq \nu \leq n)
\end{equation}
for all integers $t \geq 0$ except for the cases when $t=0$ and $b_\nu=1$.
His results can be summarized as follows.

\begin{theorem}[Chernick \cite{Chernick:1939} (1939)] \label{thm:Chernick}
For each $n \geq 3$ there exist universal forms $U_n(t)$ with computable
coefficients $a_\nu, b_\nu \in \NN$. Moreover, for fixed $n \geq 3$ and
$t \geq 0$, a universal form $U_n(t)$ represents a Carmichael number in $\CN_n$,
if each factor $a_\nu \, t + b_\nu$ is an odd prime.
\end{theorem}

\begin{remark}
Chernick required to replace $t$ by $2t$, if all coefficients $a_\nu$ and $b_\nu$
are odd; otherwise, odd values of~$t$ would cause even values of $U_n(t)$.
Actually, this already happens, if one pair $(a_\nu, b_\nu)$ consists of odd
integers. However, we explicitly left $t$ unchanged for our purpose. We fix this
problem by requiring that a factor $a_\nu \, t + b_\nu$ must be an \emph{odd}
prime instead of a prime, as stated in Theorems~\ref{thm:Chernick},
\ref{thm:Chernick2}, and~\ref{thm:main}.
\end{remark}

For the special case $n=3$ Chernick gave a general construction of $U_n(t)$,
whereas we use a more suitable formulation by introducing several definitions,
as follows.

Define the set
\[
  \RS := \set{\rr=(r_1,r_2,r_3) \in \NN^3 : r_1 < r_2 < r_3, \,
    \text{being pairwise coprime}}
\]
and the elementary symmetric polynomials for $\rr \in \RS$ as
\begin{align}
  \sigma_1(\rr) &:= r_1 + r_2 + r_3, \label{eq:sigma-1} \\
  \sigma_2(\rr) &:= r_1 r_2 + r_1 r_3 + r_2 r_3, \label{eq:sigma-2} \\
  \sigma_3(\rr) &:= r_1 r_2 r_3. \label{eq:sigma-3}
\end{align}
We implicitly use the abbreviation $\sigma_\nu$ for $\sigma_\nu(\rr)$, if there
is no ambiguity in context. For $\rr \in \RS$ define the parameter $\ell$ with
$0 \leq \ell < \sigma_3$ satisfying
\begin{equation} \label{eq:param-ell}
  \ell \equiv - \frac{\sigma_1}{\sigma_2} \pmod{\sigma_3}.
\end{equation}

One easily verifies the following parity relations for $\rr \in \RS$.
\begin{align}
\shortintertext{If $\sigma_3$ is odd, then}
  \sigma_1 \equiv \sigma_2 \equiv \sigma_3 &\equiv 1 \pmod{2};
  \label{eq:sigma-odd} \\
\shortintertext{otherwise,}
  \ell \equiv \sigma_1 \equiv \sigma_2+1 \equiv \sigma_3 &\equiv 0 \pmod{2}.
  \label{eq:sigma-even}
\end{align}

\begin{remark}
Note that congruence \eqref{eq:param-ell} is always solvable, since $\sigma_2$
is invertible $\pmods{\sigma_3}$. This will be shown by Lemma~\ref{lem:sigma-ell}.
Avoiding the expression $1/\sigma_2$, Chernick used the compatible expression
$\sigma_2^a \pmod{\sigma_3}$ with $a = \varphi(\sigma_3)-1$, where
$\varphi(\cdot)$ is Euler's totient function.
\end{remark}

With the definitions above define the forms with three factors as
\begin{equation} \label{eq:Ur-def}
  U_{\rr}(t) := \prod_{\nu=1}^{3} \, ( r_\nu \, ( \sigma_3 \, t + \ell ) + 1 )
    \quad (\rr \in \RS),
\end{equation}
allowing $\rr$ as an index in place of $n$.

\begin{theorem}[Chernick \cite{Chernick:1939} (1939)] \label{thm:Chernick2}
If \mbox{$\rr \in \RS$}, then $U_{\rr}(t)$ is a universal form.
Moreover, for fixed $t \geq 0$, $U_{\rr}(t)$ is a Carmichael number in $\CN_3$,
if each of its three factors is an odd prime.
\end{theorem}

Chernick gave some examples of $U_{\rr}(t)$, which are listed in
Table~\ref{tbl:Ur-1}. The simplest one is
\begin{equation} \label{eq:Ur-simple}
  U_{\rr}(t) = (6\,t+1)(12\,t+1)(18\,t+1) \quad (\rr = (1,2,3))
\end{equation}
as used in the introduction. The following theorem shows some unique properties
of this $U_{\rr}(t)$, compared to the case $\rr \neq (1,2,3)$.

\begin{theorem} \label{thm:Ur-prop}
Let $\rr \in \RS$ and rewrite \eqref{eq:Ur-def} as
\begin{equation} \label{eq:Ur-prod}
  U_{\rr}(t) = \prod_{\nu=1}^{3} \, (a_\nu \, t + b_\nu).
\end{equation}
Then $U_{\rr}(t)$ has the following properties for $t \in \ZZ$:

\begin{enumerate}
\item If $\rr = (1,2,3)$, then there are the equivalent properties
\[
  U_{\rr}(0) = 1, \quad \ell = 0, \andq b_\nu = 1 \quad (\nu = 1,2,3).
\]
Moreover, one has in this case
\begin{alignat*}{3}
  U_{\rr}(t) &\equiv 1 \pmod{2 \sigma_3^2}, \\
  U_{\rr}(t) &\equiv 1 \pmod{\sigma_3^3} &\quad& (t \not\equiv -1 \; \pmods{3}). \\
\shortintertext{In particular, $U_{\rr}(t)$ is odd and satisfies}
  U_{\rr}(t) &\equiv 1 \pmod{8}.
\end{alignat*}

\item If $\rr \neq (1,2,3)$, then $\ell \neq 0$, $b_\nu \neq 1$ $(\nu = 1,2,3)$, and
\begin{align*}
  U_{\rr}(0) &\equiv 1 \pmod{\sigma_3 \, \ell}, \\
  U_{\rr}(1) &\equiv 1 \pmod{\sigma_3 \, (\sigma_3 + \ell)}, \\
  U_{\rr}(t) &\equiv 1 \pmod{\sigma_3 \gcd(\sigma_3,\ell)}. \\
\shortintertext{In particular, if $\sigma_3$ is even,
then $U_{\rr}(t)$ is odd and satisfies}
  U_{\rr}(t) &\equiv 1 \pmod{4}.
\end{align*}
Otherwise, the parity of $U_{\rr}(t)$ alternates.
More precisely, if $\sigma_3$ is odd, then
\begin{alignat*}{3}
  U_{\rr}(t) &\equiv \delta(t) &&\pmod{2}, \\
  U_{\rr}(t) &\equiv 1 &&\pmod{2^{\delta(t)} \sigma_3 \gcd(\sigma_3,\ell)},
\end{alignat*}
where
\[
  \delta(t) :=
  \begin{cases}
    1, & \text{if $t \equiv \ell$ $\pmods{2}$}, \\
    0, & \text{otherwise}.
  \end{cases}
\]
\end{enumerate}
\end{theorem}

The next theorem shows the following remarkable property of $U_{\rr}(t)$.
Given any $\rr \in \RS$ we have that $U_{\rr}(t) \in \SD$ for $t \geq 2$.
Besides certain exceptions this property also holds in the case $t=1$.
More precisely, for those $t \geq 1$ in question the three factors of $U_{\rr}(t)$,
as given by~\eqref{eq:Ur-def}, already form a strict \mbox{$s$-decomposition}.
If the three factors are odd primes, then $U_{\rr}(t) \in \CN_3$ by
Theorem~\ref{thm:Chernick2}. Moreover, using the property $U_{\rr}(t) \in \SD$,
it then follows that $U_{\rr}(t) \in \CP_3$.
Thereby we arrive at our main results.

\begin{theorem} \label{thm:main}
Let $\rr \in \RS$ and define
\[
  \tau := \begin{cases}
    2, & \text{if $r_1 = 1$ and $\ell < \sigma_3 - \sigma_1$}, \\
    1, & \text{otherwise}.
  \end{cases}
\]
If $t \geq \tau$, then
\[
  U_\rr(t) = g_1 \cdot g_2 \cdot g_3 \in \SD,
\]
where the three factors are given by
\[
  g_\nu = r_\nu \, ( \sigma_3 \, t + \ell ) + 1 \quad (\nu=1,2,3)
\]
and yield a strict $s$-decomposition. Moreover, if each factor $g_\nu$ is an
odd prime, then $U_\rr(t)$ represents a primary Carmichael number, namely,
\[
  U_\rr(t) \in \CP_3.
\]
\end{theorem}

The complementary cases omitted by Theorem~\ref{thm:main}
are handled by the following theorem.

\begin{theorem} \label{thm:main2}
Let $\rr \in \RS$ and the symbols defined as in Theorem~\ref{thm:main}.
Define the integer parameter
\[
  \vartheta := \frac{\sigma_1}{r_3} + \frac{\ell \sigma_3}{r_3^2} \geq 2.
\]
For the complementary cases
\[
  U_\rr(t) = g_1 \cdot g_2 \cdot g_3 \quad (0 \leq t < \tau)
\]
the following statements hold.
\begin{enumerate}
\item If each factor $g_\nu$ is an odd prime, then $U_\rr(t) \in \CN_3$.
      Additionally,
\begin{alignat*}{3}
  U_\rr(t) &\in \CP_3, &\quad& \text{if $t=0$ and $U_\rr(t) \in \SD$}, \\
  U_\rr(t) &\notin \CP_3, && \text{if $t=1$}.
\end{alignat*}
\end{enumerate}

In particular, for $m = U_\rr(t)$ there are the following properties.
\begin{enumerate}[resume]
\item If $t=0$, then
\begin{align*}
  \vartheta = 2 &\impliest s_{g_3}(m) < g_3, \; m = g_3^2, \; g_3 = g_1 \, g_2, \\
  \vartheta > 2 &\impliest s_{g_3}(m) = g_3, \; m > g_3^2.
\end{align*}

\item If $t=1$, then $m \in \SDG$ and its $s$-decomposition
      $g_1 \cdot g_2 \cdot g_3 \in \SDG \setminus \SD$ with
      $s_{g_1}(m) = 2g_1 - 1, \, s_{g_2}(m) = g_2, \, s_{g_3}(m) = g_3$.
\end{enumerate}
\end{theorem}

\begin{remark}
To ensure the property $U_\rr(t) \in \SD$, the parameter $\tau \in \set{1,2}$
in Theorem~\ref{thm:main} cannot be improved in general. Table~\ref{tbl:Ur-val}
shows examples (taken from Tables~\ref{tbl:Ur-1} and~\ref{tbl:Ur-2}) that
satisfy the conditions of Theorem~\ref{thm:main2}. Note that for $\rr=(1,2,7)$
the decomposition $3 \cdot 5 \cdot 15 \notin \SDG$, while the value satisfies
$U_\rr(0) = 225 = 5^2 \cdot 9 \in \SD$. The case $t=0$ and $\vartheta=2$,
implying that $U_\rr(0)$ is a square, is established by a relationship between
$U_\rr(t)$ and the polygonal numbers, see Section~\ref{sec:poly}.
\end{remark}

\begin{table}[H]
\begin{center} \small
\begin{tabular}{ccrl@{$\;=\;$}rr@{$\;$}l}
  \toprule
  \multicolumn{1}{c}{$\rr$} & \multicolumn{1}{c}{$(\tau,t)$} & \multicolumn{1}{c}{$\vartheta$} &
  \multicolumn{2}{c}{value} & \multicolumn{2}{c}{decomposition}\\
  \midrule
  $(1,2,3)$  & $(1,0)$ & $2$ & $U_\rr(0)$ & $1$ & $1 \cdot 1 \cdot 1$    & $\notin \SDG$ \\
  $(1,2,7)$  & $(2,0)$ & $2$ & $U_\rr(0)$ & $225$ & $3 \cdot 5 \cdot 15$ & $\notin \SDG$ \\
  \midrule
  $(2,7,13)$ & $(1,0)$ & $6$ & $U_\rr(0)$ & $13\,833$ & $9 \cdot 29 \cdot 53$   & $\in \SDG \setminus \SD$ \\
  $(1,2,7)$  & $(2,1)$ & $2$ & $U_\rr(1)$ & $63\,393$ & $17 \cdot 33 \cdot 113$ & $\in \SDG \setminus \SD$ \\
  \midrule
  $(1,3,5)$ & $(1,0)$ & $9$  & $U_\rr(0)$ & $29\,341$  & $13 \cdot 37 \cdot 61$  & $\in \SD \cap \CP_3$ \\
  $(2,3,5)$ & $(1,0)$ & $26$ & $U_\rr(0)$ & $252\,601$ & $41 \cdot 61 \cdot 101$ & $\in \SD \cap \CP_3$ \\
  \bottomrule
\end{tabular}

\caption{Examples of $U_{\rr}(0)$ and $U_{\rr}(1)$.}
\label{tbl:Ur-val}
\end{center}
\end{table}

Table~\ref{tbl:Ur-1} reproduces the examples of $U_{\rr}(t)$ given by Chernick,
while we give further examples in Table~\ref{tbl:Ur-2}. Both tables are extended
by a third column with parameters $(\sigma_1,\sigma_2,\sigma_3,\ell,\tau)$.

\begin{table}[H]
\begin{center} \small
\begin{tabular}{ccc}
  \toprule
  \multicolumn{1}{c}{$\rr$} & \multicolumn{1}{c}{$U_{\rr}(t)$} &
  \multicolumn{1}{c}{$(\sigma_1,\sigma_2,\sigma_3,\ell,\tau)$}\\
  \midrule
  $(1,2,3)$ & $(6\,t+1)(12\,t+1)(18\,t+1)$ & $(6, 11, 6, 0, 1)$ \\
  $(1,2,5)$ & $(10\,t+7)(20\,t+13)(50\,t+31)$ & $(8, 17, 10, 6, 1)$ \\
  $(1,3,8)$ & $(24\,t+13)(72\,t+37)(192\,t+97)$ & $(12, 35, 24, 12, 1)$ \\
  $(2,3,5)$ & $(60\,t+41)(90\,t+61)(150\,t+101)$ & $(10, 31, 30, 20, 1)$ \\
  \bottomrule
\end{tabular}

\caption{Chernick's examples of $U_{\rr}(t)$.}
\label{tbl:Ur-1}
\end{center}
\end{table}

\begin{table}[H]
\begin{center} \small
\begin{tabular}{ccc}
  \toprule
  \multicolumn{1}{c}{$\rr$} & \multicolumn{1}{c}{$U_{\rr}(t)$} &
  \multicolumn{1}{c}{$(\sigma_1,\sigma_2,\sigma_3,\ell,\tau)$}\\
  \midrule
  $(1,2,7)$ & $(14\,t+3)(28\,t+5)(98\,t+15)$ & $(10, 23, 14, 2, 2)$ \\
  $(1,3,4)$ & $(12\,t+5)(36\,t+13)(48\,t+17)$ & $(8, 19, 12, 4, 1)$ \\
  $(1,3,5)$ & $(15\,t+13)(45\,t+37)(75\,t+61)$ & $(9, 23, 15, 12, 1)$ \\
  $(2,7,13)$ & $(364\,t+9)(1274\,t+29)(2366\,t+53)$ & $(22, 131, 182, 4, 1)$ \\
  \bottomrule
\end{tabular}

\caption{Further examples of $U_{\rr}(t)$.}
\label{tbl:Ur-2}
\end{center}
\end{table}

The example of a special $U_{\rr}(1) \in \CP_3$, which was used in the introduction
as $\tilde{U}_3(1)$, is shown in Table~\ref{tbl:Ur-large}. To find such an example,
the parameter \mbox{$\rr=(p_1, p_2, p_3)$} was constructed by primes that were
chosen from a finite set of primes.

\begin{table}[H]
\begin{center} \small
\begin{tabular}{c|c}
  \toprule
  $\rr$ & $(101, 199, 499)$ \\
   & \hspace*{0.15em} $(1\,012\,969\,501 \, t + 557\,673\,420)$ \\
  $U_{\rr}(t)$ & $\times \,\, (1\,995\,850\,799 \, t + 1\,098\,782\,282)$ \\
   & $\times \,\, (5\,004\,671\,099 \, t + 2\,755\,237\,982)$ \\
  $(\sigma_1,\sigma_2,\sigma_3,\ell,\tau)$ & $(799, 169\,799, 10\,029\,401, 5\,521\,519, 1)$ \\
  \bottomrule
\end{tabular}

\caption{Example of $U_{\rr}(1) \in \CP_3$.}
\label{tbl:Ur-large}
\end{center}
\end{table}

At the end of this section, we consider the case when $U_n(t)$ has $n \geq 4$
factors. Unfortunately, several computations suggest that the strong property
$U_n(t) \in \SD$, which is a necessary (but not sufficient) condition for
$U_n(t)$ to be in $\CP_n$, breaks down for $n \geq 4$.

However, it seems that a weaker property, if we replace $\SD$ by
\mbox{$\SL \setminus \SD$}, still holds in the case $n = 4$. This situation may
be confirmed by adapting the proof of Theorem~\ref{thm:main} from case $n=3$ to
$n=4$, roughly speaking.

For a provisional verification one can use Chernick's examples of $U_4(t)$
in~\cite{Chernick:1939}. On the basis of extended computations and considering
the set $\CS_4$ of exceptional Carmichael numbers, we raise the following
conjecture for the more complicated case $n=4$.

\begin{conjecture}
If $U_4(t)$ is a universal form, then $U_4(t)$ satisfies the following properties
for all sufficiently large $t$:
\begin{enumerate}
\item $U_4(t) \in \SL \setminus \SD$.
\item $U_4(t) \notin \CP_4$.
\end{enumerate}
\end{conjecture}


\section{Complementary Cases}
\label{sec:compl}

Chernick showed that any number $m \in \CN_3$ obeys a special formula, which is
intimately connected with $U_{\rr}(t)$. Actually, he defined his universal forms
thereafter. Recall the definitions of $\sigma_\nu$ and $\ell$ in
\eqref{eq:sigma-1} -- \eqref{eq:param-ell}. The result can be stated as follows.

\begin{theorem}[Chernick \cite{Chernick:1939} (1939)] \label{thm:Chernick3}
If $m \in \CN_3$, then there exists a unique $\rr \in \RS$ such that
\[
  m = (r_1 \, u + 1)(r_2 \, u + 1)(r_3 \, u + 1),
\]
where $u$ is an even positive integer. More precisely,
if $m = p_1 \cdot p_2 \cdot p_3$ with odd primes $p_1 < p_2 < p_3$, then
\begin{align*}
  u &= \gcd( p_1 - 1, p_2 - 1, p_3 - 1 ) \\
\shortintertext{and}
  \rr &= \mleft( \frac{p_1 - 1}{u}, \frac{p_2 - 1}{u}, \frac{p_3 - 1}{u} \mright).
\end{align*}
Moreover,
\[
  m = U_{\rr}(t),
\]
where $t \geq 0$ is an integer satisfying $u = \sigma_3 \, t + \ell$.
\end{theorem}

As a result of Theorem~\ref{thm:main}, we have for any $\rr \in \RS$ that
\[
  U_\rr(t) \in \SD \quad (t \geq \tau),
\]
where $\tau \in \set{1,2}$. Moreover,
\begin{equation} \label{eq:Ur-CP}
  U_\rr(t) = p_1 \cdot p_2 \cdot p_3 \impliesq
    U_\rr(t) \in \CP_3 \quad (t \geq \tau),
\end{equation}
when $p_1$, $p_2$, and $p_3$ are odd primes.

In the complementary cases $0 \leq t < \tau$, Theorem~\ref{thm:main2} predicts
that $U_{\rr}(t) \in \CP_3$ can only happen when $t=0$. Table~\ref{tbl:except}
shows the first of those values with parameters $\rr$ and $(\tau,t)$.

The remaining values, where $U_\rr(t) \notin \SD$ for $0 \leq t < \tau$, can be
viewed as exceptions. The next theorem clarifies these cases in the context of
Carmichael numbers $m \in \CN_3 \setminus \CP_3$.

\begin{theorem} \label{thm:except}
If $m \in \CN_3 \setminus \CP_3$, then we have
\[
  m \in (\SDG \cap \SL) \setminus \SD,
\]
where the greatest prime divisor $p$ of $m$ satisfies
\begin{equation} \label{eq:cond-p}
  s_p(m) = p.
\end{equation}
Moreover, there exist a unique $\rr \in \RS$, as defined in
Theorem~\ref{thm:Chernick3}, and an integer $t$ such that
\[
  m = U_{\rr}(t)
\]
with $0 \leq t < \tau$, where $\tau \in \set{1,2}$ is defined as in
Theorem~\ref{thm:main}.

In the case $(\tau,t) = (2,1)$, property \eqref{eq:cond-p} also holds for the
second greatest prime divisor $p$ of $m$.
\end{theorem}

\begin{remark}
For several numbers $m = p_1 \cdot p_2 \cdot p_3 \in \CN_3 \setminus \CP_3$
in the case $t = 0$, property \eqref{eq:cond-p} holds, as in the case
$(\tau,t) = (2,1)$, also for $p_2$. However, the first example occurs for
\[
  m = 6\,709\,788\,961 = 337 \cdot 421 \cdot 47\,293,
\]
where \eqref{eq:cond-p} does not hold for $p_2$, as verified by
\[
  s_{p_1}(m) = p_1, \quad s_{p_2}(m) = 2p_2-1, \quad s_{p_3}(m) = p_3.
\]
\end{remark}

The first numbers $m \in \CN_3 \setminus \CP_3$ with parameters $\rr$ and $(\tau,t)$
are listed in Table~\ref{tbl:except2}. By Theorem~\ref{thm:except} such numbers
can be represented by $U_{\rr}(t)$ with certain $\rr \in \RS$ only in the cases
$0 \leq t \leq 1$, while for any $\rr \in \RS$ each $U_{\rr}(t)$ represents only
primary Carmichael numbers for $t \geq 2$ when satisfying~\eqref{eq:Ur-CP}.

Supported by computations of the ratio $C'_3(x)/C_3(x)$ in Table~\ref{tbl:distrib},
Dickson's conjecture, applied to $U_{\rr}(t)$, implies the following conjecture.

\begin{conjecture}
We have
\[
  \lim_{x \, \to \, \infty} \, \frac{C'_3(x)}{C_3(x)} \, = \, 1.
\]
\end{conjecture}

Due to the very special properties of the primary Carmichael numbers, one may
initially believe that these numbers play a minor role when comparing the
distributions of $C(x)$ and $C'(x)$ in Table~\ref{tbl:distrib}.
Only a closer look at the case of $3$-factor Carmichael numbers reveals
that primary Carmichael numbers play admittedly a central role in that context.

\begin{table}[H]
\begin{center} \small
\begin{tabular}{rcccrcc}
  \toprule
  \multicolumn{1}{c}{$m$} & \multicolumn{1}{c}{$\rr$} & \multicolumn{1}{c}{$(\tau,t)$} &
  \multicolumn{1}{c}{\hspace*{1em}} &
  \multicolumn{1}{c}{$m$} & \multicolumn{1}{c}{$\rr$} & \multicolumn{1}{c}{$(\tau,t)$} \\
  \midrule
  $2821$ & $(1,2,5)$ & $(1,0)$           & & $14\,469\,841$ & $(4,21,29)$ & $(1,0)$ \\
  $29\,341$ & $(1,3,5)$ & $(1,0)$        & & $15\,247\,621$ & $(1,3,23)$ & $(1,0)$ \\
  $46\,657$ & $(1,3,8)$ & $(1,0)$        & & $15\,829\,633$ & $(1,13,16)$ & $(2,0)$ \\
  $252\,601$ & $(2,3,5)$ & $(1,0)$       & & $17\,236\,801$ & $(5,7,18)$ & $(1,0)$ \\
  $1\,193\,221$ & $(1,2,21)$ & $(1,0)$   & & $17\,316\,001$ & $(1,3,40)$ & $(2,0)$ \\
  $1\,857\,241$ & $(1,6,11)$ & $(2,0)$   & & $29\,111\,881$ & $(3,4,7)$ & $(1,0)$ \\
  $5\,968\,873$ & $(1,3,26)$ & $(2,0)$   & & $31\,405\,501$ & $(1,9,10)$ & $(1,0)$ \\
  $6\,868\,261$ & $(1,5,18)$ & $(2,0)$   & & $34\,657\,141$ & $(19,42,43)$ & $(1,0)$ \\
  $7\,519\,441$ & $(1,6,19)$ & $(2,0)$   & & $35\,703\,361$ & $(5,23,176)$ & $(1,0)$ \\
  $10\,024\,561$ & $(7,27,52)$ & $(1,0)$ & & $37\,964\,809$ & $(2,7,17)$ & $(1,0)$ \\
  \bottomrule
\end{tabular}

\caption{First numbers $m = U_{\rr}(0) \in \CP_3$.}
\label{tbl:except}
\end{center}
\end{table}

\begin{table}[H]
\begin{center} \small
\begin{tabular}{rcccrcc}
  \toprule
  \multicolumn{1}{c}{$m$} & \multicolumn{1}{c}{$\rr$} & \multicolumn{1}{c}{$(\tau,t)$} &
  \multicolumn{1}{c}{\hspace*{1em}} &
  \multicolumn{1}{c}{$m$} & \multicolumn{1}{c}{$\rr$} & \multicolumn{1}{c}{$(\tau,t)$} \\
  \midrule
  $561$ & $(1,5,8)$ & $(2,0)$       & & $314\,821$ & $(1,5,33)$ & $(2,0)$ \\
  $1105$ & $(1,3,4)$ & $(1,0)$      & & $334\,153$ & $(3,7,68)$ & $(1,0)$ \\
  $2465$ & $(1,4,7)$ & $(2,0)$      & & $410\,041$ & $(5,9,17)$ & $(1,0)$ \\
  $6601$ & $(3,11,20)$ & $(1,0)$    & & $530\,881$ & $(1,8,35)$ & $(2,0)$ \\
  $8911$ & $(1,3,11)$ & $(2,0)$     & & $1\,024\,651$ & $(1,11,15)$ & $(2,0)$ \\
  $10\,585$ & $(1,7,18)$ & $(2,0)$  & & $1\,461\,241$ & $(1,2,15)$ & $(2,1)$ \\
  $15\,841$ & $(1,5,12)$ & $(2,0)$  & & $1\,615\,681$ & $(1,9,16)$ & $(2,0)$ \\
  $52\,633$ & $(1,12,17)$ & $(2,0)$ & & $1\,909\,001$ & $(2,5,23)$ & $(1,0)$ \\
  $115\,921$ & $(1,3,20)$ & $(2,0)$ & & $2\,508\,013$ & $(2,3,23)$ & $(1,0)$ \\
  $162\,401$ & $(2,5,29)$ & $(1,0)$ & & $3\,057\,601$ & $(1,5,8)$ & $(2,1)$ \\
  \bottomrule
\end{tabular}

\caption{First numbers $m = U_{\rr}(t) \in \CN_3 \setminus \CP_3$.}
\label{tbl:except2}
\end{center}
\end{table}

\vspace*{8pt}


\section{Proofs of Theorems \pref{thm:s-decomp} and \pref{thm:SL-elem}}
\label{sec:proofs-1}

Recall the definitions of Section~\ref{sec:decomp}.

\begin{lemma} \label{lem:estim-g-m}
Let $g, m \in \NN$. If $g \mid m$ and $s_g(m) \geq g$, then
\[
  1 < g < m^{1/(\ord_{g}(m)+1)} \leq \sqrt{m}.
\]
\end{lemma}

\begin{proof}
Since $s_1(m) = 0$ and $s_m(m) = 1$, the conditions $g \mid m$ and $s_g(m) \geq g$
imply that $g$ is a proper divisor of $m$, and therefore $1 < g < m$.
Letting $e = \ord_g(m) \geq 1$, we can write $m = g^e \, m'$ with $g \nmid m'$.
Since $m' < g$ would imply $s_g(m) = s_g(m') < g$, it follows that $m' > g$.
As a consequence, we obtain $m > g^{e+1} \geq g^2$, showing the result.
\end{proof}

\begin{proof}[Proof of Theorem~\ref{thm:s-decomp}]
Let $m \in \SDG$. We have to show five parts.

(i).~Since $m = g_1^{e_1}$ with $e_1 \geq 1$ yields $s_{g_1}(m) = 1$,
$m$ must have at least two factors in its $s$-decomposition.
Next we consider the prime factorization $m = p^e$ with $e \geq 1$.
For any factor $g = p^\nu$ of $m$ with $1 \leq \nu \leq e$,
we infer that $s_g(m) < g$. Thus, $m$ has no $s$-decomposition in this case.
Finally, $m$ must have at least two prime factors.

(ii).~Assume that $m = g_1 \cdot g_2$ is an $s$-decomposition.
With $g_1 < g_2$ we then obtain that $s_{g_2}( m ) = s_{g_2}( g_1 ) < g_2$,
getting a contradiction. This implies that $m = g_1^{e_1} \cdot g_2^{e_2}$
must satisfy $e_1 + e_2 \geq 3$.

(iii).~We have $m = g_1 \cdot g_2 \cdot g_3$, where all $g_\nu$ are odd primes.
Assume that the $s$-decomposition of $m$ is not unique. Then by part~(i) we would
have $m = \tilde{g}_1 \cdot \tilde{g}_2$ as a further $s$-decomposition, where
$\tilde{g}_1$ is a prime and $\tilde{g}_2$ is a product of two primes, or vice
versa. But this contradicts part~(ii). Additionally, If $m \in \SD$, then $m$
also satisfies the condition to be in $\CP$, and thus $m \in \CP_3$.

(iv).~We have the inclusions $\SH \subset \SDG$ and $\SD \subset \SDG$.
If $m$ has the \mbox{$s$-decomposition} $g_1 \dotsm g_n$ with $n \geq 3$
factors, where $g_\nu$ are odd primes, then $m \in \SH$ by definition.
Similarly, if $g_1 \dotsm g_n \in \SD$ is a strict \mbox{$s$-decomposition},
then $m \in \CP_n$.

(v).~The exponent $e_\nu$ of each factor $g_\nu$ of the $s$-decomposition of
$m$ satisfies $e_\nu \leq \ord_{g_\nu}(m)$. The result then follows from
Lemma~\ref{lem:estim-g-m}.

This completes the proof of the theorem.
\end{proof}

\begin{proof}[Proof of Theorem~\ref{thm:SL-elem}]
Let $m \in \SLG$ and $g \mid m$ with $s_g(m) \geq g$.
We have to show three parts.

(i).~Assume that $m = p^e$ with $e \geq 1$. Then $g = p^\nu$ with
$1 \leq \nu \leq e$. Since $s_g(m) < g$, we get a contradiction.
Therefore $m$ must have at least two prime factors.

(ii).~We have $m \in \CN_3 \subset \SDG$. From Theorems~\ref{thm:CN-prop}
and~\ref{thm:s-decomp}~(iii), it follows that $m = p_1 \cdot p_2 \cdot p_3$ is
a unique $s$-decomposition, implying that $g$ is an odd prime.

(iii).~The inequalities follow from Lemma~\ref{lem:estim-g-m}, finishing the proof.
\end{proof}


\section{Proofs of Theorems \pref{thm:Ur-prop}, \pref{thm:main},
\pref{thm:main2}, and \pref{thm:except}}
\label{sec:proofs-2}

Let $\ZZ_p$ be the ring of $p$-adic integers, $\QQ_p$ be the field of $p$-adic
numbers, and $\pval_p(s)$ be the $p$-adic valuation of $s \in \QQ_p$.
As a basic property of $p$-adic numbers, we have
\begin{equation} \label{eq:pval}
  \pval_p(x + y) \geq \min( \pval_p(x), \pval_p(y) ) \quad (x,y \in \QQ_p),
\end{equation}
where equality holds if $\pval_p(x) \neq \pval_p(y)$
(see \cite[Sec.~1.5, pp.~36--37]{Robert:2000}).

For $x \in \RR$ we write $x = \intpart{x} + \fracpart{x}$, where $\intpart{x}$
denotes the integer part, and $0 \leq \fracpart{x} < 1$ denotes the fractional
part. Recall the definitions of $\sigma_\nu$ and $\ell$ in
\eqref{eq:sigma-1} -- \eqref{eq:param-ell}. We set $J := \set{1,2,3}$ and use
$j \in J$ as an index, mainly in the context of $\rr \in \RS$.
Before proving the theorems, we need several lemmas.

\begin{lemma} \label{lem:sigma-equal}
Let $\rr \in \RS$. If $\rr = (1,2,3)$, then $\sigma_3 = \sigma_1 = 6$;
otherwise, $\sigma_3 > \sigma_1 > 6$.
\end{lemma}

\begin{proof}
First we consider the triple $(1,2,r_3) \in \RS$ with $r_3 \geq 3$. We then
obtain that $\sigma_3 = 2 r_3 \geq 3 + r_3 = \sigma_1 \geq 6$, where equality
can only hold for $r_3 = 3$, respectively, $(1,2,3) \in \RS$. This shows this
case. Since $r_1 < r_2 < r_3$ for $\rr \in \RS$, there remains the case where
$r_1 r_2 \geq 3$. It follows that $\sigma_3 \geq 3 r_3 > 3 r_3 - 3 \geq \sigma_1 > 6$,
completing the proof.
\end{proof}

\begin{lemma} \label{lem:sigma-short}
Let $\rr \in \RS$ and $j \in J$. Define
\begin{equation} \label{eq:sigma-short}
  \sigma'_3 := \frac{\sigma_3}{r_j} \andq \sigma'_1 := \sigma_1 - r_j.
\end{equation}
If $r_j \geq 2$, then
\begin{equation} \label{eq:sigma-short2}
  \sigma_2 \equiv \sigma'_3 \not\equiv 0 \andq
  \sigma_1 \equiv \sigma'_1 \pmod{r_j}.
\end{equation}
\end{lemma}

\begin{proof}
Let $r_j \geq 2$. One observes by \eqref{eq:sigma-2} and \eqref{eq:sigma-3} that
\[
  \sigma_2 \equiv \sigma'_3 \not\equiv 0 \pmod{r_j},
\]
since the integers $r_1$, $r_2$, and $r_3$ are pairwise coprime.
The congruence $\sigma_1 \equiv \sigma'_1 \pmod{r_j}$ follows from the definition.
\end{proof}

\begin{lemma} \label{lem:sigma-ell}
Let $\rr \in \RS$ and the parameter $\ell$ be defined as in~\eqref{eq:param-ell} by
\[
  \ell \equiv - \frac{\sigma_1}{\sigma_2} \pmod{\sigma_3},
\]
where $0 \leq \ell < \sigma_3$. The congruence is always solvable,
since $\sigma_2$ is invertible $\pmods{\sigma_3}$. In particular,
\begin{equation} \label{eq:ell-0}
  \ell = 0 \iffq \rr = (1,2,3).
\end{equation}
\end{lemma}

\begin{proof}
By \eqref{eq:sigma-short2} we have for $j \in J$ and $r_j \geq 2$ that
\begin{equation} \label{eq:sigma-2-mod}
  \sigma_2 \not\equiv 0 \pmod{r_j}.
\end{equation}
Note that in case $r_1 = 1$ we have to consider $\sigma_3 = r_2 r_3$ with two
factors instead of $\sigma_3 = r_1 r_2 r_3$. Since the integers $r_j$ are
pairwise coprime, it follows that $\sigma_2$ is invertible $\pmods{\sigma_3}$
by~\eqref{eq:sigma-2-mod}.
Therefore, $\ell = 0$ if and only if $\sigma_1 \equiv 0 \pmod{\sigma_3}$.
As $\sigma_3 \geq \sigma_1 > 0$ and $\sigma_3 = \sigma_1$ if and only if $\rr = (1,2,3)$
by Lemma~\ref{lem:sigma-equal}, relation~\eqref{eq:ell-0} follows.
\end{proof}

\begin{lemma} \label{lem:sigma-eta}
If $\rr \in \RS$ and $j \in J$, then
\begin{equation} \label{eq:eta-def}
  \eta := \frac{\sigma_1}{r_j} + \frac{\ell \sigma_3}{r_j^2} \geq 2
\end{equation}
is an integer, and the bound is sharp. In particular, $\eta = 2$ holds for $j=3$
in both cases $\ell = 0$ and $\ell \neq 0$ by $\rr = (1,2,3)$ and $\rr = (1,2,7)$,
respectively.
\end{lemma}

\begin{proof}
If $r_j = 1$, then $\eta$ is integral. Assume that $r_j \geq 2$. Using
Lemmas~\ref{lem:sigma-short} and~\ref{lem:sigma-ell}, we obtain
\begin{equation} \label{eq:ell-congr}
  \ell \equiv - \frac{\sigma_1}{\sigma_2}
    \equiv - \frac{\sigma'_1}{\sigma'_3} \pmod{r_j}.
\end{equation}
For the reduced numerator of $\eta$ we then infer that
\[
  \sigma_1 + \ell \sigma'_3 \equiv
  \sigma'_1 - \frac{\sigma'_1}{\sigma'_3} \sigma'_3 \equiv 0 \pmod{r_j},
\]
implying that $\eta$ is integral. For any $r_j \geq 1$, we have
$\eta \geq \sigma_1 / r_j = 1 + \sigma'_1 / r_j > 1$, so $\eta \geq 2$.
In particular, one computes $\eta = 2$ for $r_3$ by taking $\rr = (1,2,3)$ and
$\rr = (1,2,7)$ from Tables~\ref{tbl:Ur-1} and~\ref{tbl:Ur-2}, respectively.
Both examples incorporate the cases $\ell = 0$ and $\ell \neq 0$.
This completes the proof.
\end{proof}

\begin{lemma} \label{lem:sigma-alpha}
Let $\rr \in \RS$ and $j \in J$ where $r_j \geq 2$. Define
\[
  \alpha := \frac{\sigma_3}{r_j^3} \andq
  \beta  := \frac{\sigma_3}{r_j^3} - \frac{\sigma_1}{r_j} + 1.
\]
Then $\alpha, \beta, \alpha+\beta \in \ZZ/r_j^2 \setminus \ZZ$ are fractions.
\end{lemma}

\begin{proof}
By~\eqref{eq:sigma-short} rewrite $\alpha$ and $\beta$ as
\begin{equation} \label{eq:alpha-short}
  \alpha = \frac{\sigma'_3}{r_j^2} \andq
  \beta  = \frac{\sigma'_3}{r_j^2} - \frac{\sigma'_1}{r_j}.
\end{equation}
Obviously, we have $\alpha, \beta, \alpha+\beta \in \ZZ/r_j^2$.
As $r_j \geq 2$, we show that $\alpha, \beta, \alpha+\beta \notin \ZZ$.
Let $p$ be a prime divisor of $r_j$ and $e = \pval_p(r_j) \geq 1$.
Since $\sigma'_3$ and $r_j$ are coprime, it follows that
$\pval_p(\alpha) = -2e < 0$ and thus $\alpha \notin \ZZ$. In the same way,
we infer by~\eqref{eq:pval} that $\alpha - \sigma'_1/r_j = \beta \notin \ZZ$,
since $\pval_p(\alpha) < \pval_p(\sigma'_1/r_j) = \pval_p(\sigma'_1) - e$.
Next we consider
\[
  \alpha + \beta = \frac{2\sigma'_3}{r_j^2} - \frac{\sigma'_1}{r_j},
\]
where we distinguish between two cases as follows.

\Case{$p \geq 3$} From $\pval_p(2\alpha) = \pval_p(\alpha) < \pval_p(\sigma'_1/r_j)$
and using~\eqref{eq:pval}, we derive that $\alpha + \beta \notin \ZZ$.

\Case{$p = 2$} We have that $r_j$ is even. Due to $\rr \in \RS$ and the $r_\nu$
being pairwise coprime, $\sigma'_3$ and $\sigma'_1$ must be odd and even,
respectively. Hence, $\pval_p(2\alpha) = 1-2e < 0$, while
$\pval_p(\sigma'_1/r_j) \geq 1-e$. By~\eqref{eq:pval} we get
$\alpha + \beta \notin \ZZ$.

This completes the proof.
\end{proof}

\begin{lemma} \label{lem:sigma-theta}
Let $\rr \in \RS$ and $j \in J$ where $r_j \geq 2$. Let $\alpha$ and $\beta$ be
defined as in Lemma~\ref{lem:sigma-alpha}, and
$g = r_j \, ( \sigma_3 \, t + \ell ) + 1$ with $t \in \ZZ$. Define
\[
  \theta := \fracpart{\alpha} \, g - \beta.
\]
There are the following properties:
\begin{enumerate}
\item If $t \in \ZZ$, then $\theta \in \ZZ$.

\item If $t \geq 1$, then there are the inequalities
\begin{equation} \label{eq:g-inequal}
  g > \theta > 1 + \intpart{\alpha}.
\end{equation}

\item If $\rr \neq (1,2,3)$, $j = 3$, and $t=0$, then
\begin{equation} \label{eq:g-inequal2}
  g > \theta \geq 1.
\end{equation}
\end{enumerate}
\end{lemma}

\begin{proof}
We implicitly use the definitions of \eqref{eq:sigma-short} and
\eqref{eq:alpha-short}. We have to show three parts.

(i).~As $r_j \geq 2$ and $t \in \ZZ$, we obtain by \eqref{eq:ell-congr} that
\begin{equation} \label{eq:g-congr}
  \frac{g-1}{r_j} \equiv \ell \equiv - \frac{\sigma'_1}{\sigma'_3} \pmod{r_j}.
\end{equation}
Since $\alpha = \intpart{\alpha} + \fracpart{\alpha}$, it suffices to show that
$\alpha g - \beta \in \ZZ$. We then infer that
\begin{equation} \label{eq:g-equal}
  \alpha g - \beta = \frac{\sigma'_3 \, (g - 1)}{r_j^2} + \frac{\sigma'_1}{r_j}.
\end{equation}
For the latter numerator in reduced form, it follows from \eqref{eq:g-congr} that
\[
  \sigma'_3 \frac{g-1}{r_j} + \sigma'_1 \equiv
    - \sigma'_3 \frac{\sigma'_1}{\sigma'_3} + \sigma'_1 \equiv 0 \pmod{r_j},
\]
implying that $\theta \in \ZZ$.

(ii).~We consider the inequalities~\eqref{eq:g-inequal}. First we show for
$t \geq 1$ that
\[
  g > \fracpart{\alpha} \, g - \beta,
\]
or equivalently that
\[
  (1 - \fracpart{\alpha}) \, g > -\beta.
\]
Note that $\beta$ can be negative, so this inequality is not trivial.
Since by Lemma~\ref{lem:sigma-alpha} $\alpha \in \ZZ/r_j^2 \setminus \ZZ$
is a fraction, we obtain that
\begin{equation} \label{eq:g-alpha-1}
  1 - \fracpart{\alpha} \geq \frac{1}{r_j^2}.
\end{equation}
For $t \geq 1$ we have
\begin{equation} \label{eq:g-alpha-2}
  g > r_j \, \sigma_3 \, t = r_j^4 \, \alpha \, t.
\end{equation}
Combining both inequalities above, we deduce that
\begin{equation} \label{eq:g-alpha-3}
  (1 - \fracpart{\alpha}) \, g > r_j^2 \, \alpha \, t.
\end{equation}
Therefore, we show the following inequality
\[
  r_j^2 \, \alpha \, t > -\beta = -\alpha + \frac{\sigma'_1}{r_j}.
\]

Let $i, k \in J \setminus \set{j}$ be the other two indices complementary to $j$.
Then the above inequality becomes
\begin{equation} \label{eq:t-estim}
  t > \frac{1}{r_j} \left( -\frac{1}{r_j} + \frac{\sigma'_1}{\sigma'_3} \right)
    = \frac{1}{r_j} \left( \frac{1}{r_i} + \frac{1}{r_k} - \frac{1}{r_j} \right)
    =: \mathfrak{f}(r_j).
\end{equation}

Since $r_i, r_k \geq 1$ but $r_i \neq r_k$, we can use the estimate
\[
  \mathfrak{g}(r_j) := \frac{1}{r_j} \left( 2 - \frac{1}{r_j} \right)
    > \mathfrak{f}(r_j) \quad (r_j \geq 2).
\]
It is easy to see that $\mathfrak{g}(r_j)$ is strictly decreasing for $r_j \geq 2$.
Hence, $\mathfrak{g}(2) = 3/4 > \mathfrak{f}(r_j)$ for $r_j \geq 2$, implying
that \eqref{eq:t-estim} holds for $t \geq 1$.
Finally, putting all together yields for $t \geq 1$ that
\[
  (1 - \fracpart{\alpha}) \, g > r_j^2 \, \alpha \, t > -\beta.
\]

Now we show for $t \geq 1$ that
\[
  \fracpart{\alpha} \, g - \beta > 1 + \intpart{\alpha}.
\]
Since both sides of the above inequality lie in $\ZZ$, we can also write
\[
  \fracpart{\alpha} \, g > 1 + \alpha + \beta.
\]
By the same arguments, the inequalities \eqref{eq:g-alpha-1} and
\eqref{eq:g-alpha-3} are also valid for $\fracpart{\alpha}$ in place of
$1 - \fracpart{\alpha}$. In view of \eqref{eq:g-alpha-2}, we then have
\[
  \fracpart{\alpha} \, g > r_j^2 \, \alpha \, t.
\]
Hence, we proceed in showing that
\[
  r_j^2 \, \alpha \, t > 1 + \alpha + \beta = 1 + 2\alpha - \frac{\sigma'_1}{r_j}.
\]
This turns into
\begin{equation} \label{eq:t-estim2}
  t > \frac{1}{\sigma'_3} + \frac{2}{r_j^2} - \frac{\sigma'_1}{r_j \sigma'_3}
    = A + B - C =: S.
\end{equation}
Since $\sigma'_3 \geq 2$ and $r_j \geq 2$, we obtain the estimates
\[
  A \leq \tfrac{1}{2}, \quad B \leq \tfrac{1}{2}, \andq C > 0.
\]
As a consequence, we infer that $S < 1$, and thus \eqref{eq:t-estim2} holds for
$t \geq 1$. Again, putting all together yields for $t \geq 1$ that
\[
  \fracpart{\alpha} \, g > r_j^2 \, \alpha \, t > 1 + \alpha + \beta,
\]
finally showing the inequalities \eqref{eq:g-inequal}.

(iii).~We consider the case where $\rr \neq (1,2,3)$, $j = 3$, and $t=0$.
Therefore $r_j \geq 4$, and $\ell \geq 1$ by Lemma~\ref{lem:sigma-ell}.
Since $r_1 < r_2 < r_3$, we have $\alpha = \sigma'_3/r_j^2 < 1$ and so
$\fracpart{\alpha} = \alpha$.
By $\alpha g - \beta \in \ZZ$ the inequalities \eqref{eq:g-inequal2} become
\[
  g-1 \geq \alpha g - \beta \geq 1
\]
where
\[
  g = r_j \, \ell + 1.
\]
From \eqref{eq:g-equal} we deduce that
\begin{equation} \label{eq:g-alpha-beta}
  \alpha g - \beta = \frac{\sigma'_3 \, \ell + \sigma'_1}{r_j} > 0,
\end{equation}
implying that $\alpha g - \beta \geq 1$. There remains to show that
$g-1 \geq \alpha g - \beta$. After dividing by $g-1 = r_j \, \ell$, we obtain
\begin{equation} \label{eq:s13-inequal}
  1 \geq \frac{\sigma'_3}{r_j^2} + \frac{\sigma'_1}{r_j^2 \, \ell}.
\end{equation}
Since $\ell \geq 1$, we continue with
\[
  S' := \frac{\sigma'_3 + \sigma'_1}{r_j^2}.
\]

From $r_j > 3$ and using the inequalities
\begin{align*}
  (r_j-1)+(r_j-2) &\geq \sigma'_1, \\
  (r_j-1)(r_j-2)  &\geq \sigma'_3,
\end{align*}
we infer that
\[
  S' \leq \frac{r_j^2 - r_j - 1}{r_j^2} = 1 - \frac{1}{r_j} - \frac{1}{r_j^2} < 1,
\]
implying that \eqref{eq:s13-inequal} holds and so $g-1 \geq \alpha g - \beta$.
This finally shows the inequalities \eqref{eq:g-inequal2}, completing the proof.
\end{proof}

Now we are ready to give the proofs of the theorems.

\begin{proof}[Proof of Theorem~\ref{thm:Ur-prop}]
Let $\rr \in \RS$. By \eqref{eq:Ur-def} and \eqref{eq:Ur-prod} we consider
\begin{equation} \label{eq:Ur-prods}
  U_{\rr}(t) = \prod_{j=1}^{3} \, ( r_j \, ( \sigma_3 \, t + \ell ) + 1 )
    = \prod_{j=1}^{3} \, (a_j \, t + b_j).
\end{equation}
Expanding the first product of~\eqref{eq:Ur-prods} yields
\begin{equation} \label{eq:Ur-sum}
  U_{\rr}(t) - 1 = \sum_{j=1}^{3} \, \sigma_j \, ( \sigma_3 \, t + \ell )^j.
\end{equation}
We have to show two parts.

(i).~Comparing both products of \eqref{eq:Ur-prods}, we infer that
\[
  U_{\rr}(0) = 1 \iffq \ell = 0 \iffq b_j = 1 \quad (j \in J),
\]
and from Lemma~\ref{lem:sigma-ell} it follows that
$\ell = 0$ if and only if $\rr = (1,2,3)$.

Now let $\rr = (1,2,3)$. We have $(\sigma_1,\sigma_2,\sigma_3) = (6,11,6)$.
Since $\ell = 0$ and $\sigma_1 = \sigma_3$, we deduce from \eqref{eq:Ur-sum} that
\[
  U_{\rr}(t) - 1 = \sigma_3^2 A(t) \withq A(t)
    := t \, ( \sigma_3^2 \, t^2 + \sigma_2 \, t + 1 ).
\]
For any $t \in \ZZ$ we obtain
\begin{align*}
  A(t) &\equiv t \, (1+t) \equiv 0 \pmod{2},
\shortintertext{while only for $t \not\equiv -1$ $\pmods{3}$ we have}
  A(t) &\equiv t \, (1-t) \equiv 0 \pmod{3}.
\end{align*}
This finally implies that
\begin{align}
  2\sigma_3^2 &\mid U_{\rr}(t) - 1, \label{eq:Ur-div-1} \\
\shortintertext{and if $t \not\equiv -1$ $\pmods{3}$ that}
  \sigma_3^3 &\mid U_{\rr}(t) - 1, \nonumber
\end{align}
implying the two claimed congruences. From~\eqref{eq:Ur-div-1} we then derive that
\[
  U_{\rr}(t) \equiv 1 \pmod{8}.
\]
Thus, $U_{\rr}(t)$ is odd for all $t \in \ZZ$.

(ii).~Let $\rr \neq (1,2,3)$. Then we have $0 < \ell < \sigma_3$ and
by~\eqref{eq:Ur-prods} that $b_j \neq 1$ $(j \in J)$. Using the substitution
$\lambda = \sigma_3 \, t + \ell \neq 0$ for any $t \in \ZZ$,
we obtain by \eqref{eq:Ur-sum} that
\begin{equation} \label{eq:Ur-B}
  U_{\rr}(t) - 1 = \lambda \, B(t) \withq B(t)
    := \sigma_3 \, \lambda^2 + \sigma_2 \, \lambda + \sigma_1.
\end{equation}
Furthermore, it follows from Lemma~\ref{lem:sigma-ell} that
\[
  B(t) \equiv \sigma_2 \, \ell + \sigma_1 \equiv 0 \pmod{\sigma_3}.
\]
Hence, we infer that
\begin{equation} \label{eq:Ur-div-2}
  \sigma_3 \, \lambda \mid U_{\rr}(t) - 1,
\end{equation}
where $\lambda = \ell$ if $t=0$, $\lambda = \sigma_3+\ell$ if $t=1$, and
$\gcd(\sigma_3,\ell) \mid \lambda$ in any case.
This implies the claimed congruences
\begin{align}
  U_{\rr}(0) &\equiv 1 \pmod{\sigma_3 \, \ell}, \nonumber \\
  U_{\rr}(1) &\equiv 1 \pmod{\sigma_3 \, (\sigma_3 + \ell)}, \nonumber \\
  U_{\rr}(t) &\equiv 1 \pmod{\sigma_3 \gcd(\sigma_3,\ell)}. \label{eq:Ur-congr}
\end{align}

If $\sigma_3$ is even, then \eqref{eq:sigma-even} implies that $2 \mid \ell$,
and so $2 \mid \lambda$. We then derive from~\eqref{eq:Ur-div-2} that
\[
  U_{\rr}(t) \equiv 1 \pmod{4}
\]
and $U_{\rr}(t)$ is odd for all $t \in \ZZ$.

Otherwise, $\sigma_3$ is odd. In this case it follows from~\eqref{eq:sigma-odd}
that $\sigma_1$ and $\sigma_2$ are also odd.
With that we infer from~\eqref{eq:Ur-B} that
\[
  B(t) \equiv 1 \pmod{2},
\]
regardless of the parity of $\lambda$, and therefore valid for all $t \in \ZZ$.
Moreover, \eqref{eq:Ur-B} then implies that
\[
  U_{\rr}(t) \equiv 1 + \lambda \equiv 1 + t + \ell \equiv \delta(t) \pmod{2},
\]
where $\delta(t) = 1$ if $t \equiv \ell$ $\pmods{2}$, and $\delta(t) = 0$
otherwise. This shows the alternating parity of $U_{\rr}(t)$.
If $\delta(t) = 1$, then
\[
  U_{\rr}(t) \equiv 1 \pmod{2}.
\]
Together with \eqref{eq:Ur-congr}, since $\sigma_3 \gcd(\sigma_3,\ell)$ is odd,
we finally achieve
\[
  U_{\rr}(t) \equiv 1 \pmod{2^{\delta(t)} \sigma_3 \gcd(\sigma_3,\ell)},
\]
being compatible with the case $\delta(t)=0$.
This completes the proof of the theorem.
\end{proof}

\begin{proof}[Proof of Theorem~\ref{thm:main}]
Let $\rr \in \RS$ and $t \geq 0$ be an integer.
As defined in \eqref{eq:Ur-def}, write
\[
  U_\rr(t) = g_1 \cdot g_2 \cdot g_3,
\]
where the three factors are given by
\begin{equation} \label{eq:g-j-def}
  g_j = r_j \, ( \sigma_3 \, t + \ell ) + 1 \quad (j \in J)
\end{equation}
and $0 \leq \ell < \sigma_3$ by \eqref{eq:param-ell}.

Theorem~\ref{thm:Chernick2} states that $U_\rr(t)$ is a universal form.
We briefly write
\begin{equation} \label{eq:m-prod-g}
  m = g_1 \cdot g_2 \cdot g_3,
\end{equation}
keeping in mind that $m$ and the $g_j$ depend on $t$.

We have to determine an integer $\tau \in \set{1,2}$ as claimed such that
the strict sum-of-digits condition holds for $t \geq \tau$ as follows.
\begin{equation} \label{eq:s-conds-m}
  s_{g_j}( m ) = g_j \quad (j \in J).
\end{equation}
In this case, the right-hand side of \eqref{eq:m-prod-g} provides a strict
$s$-decomposition of~$m$, and thus
\[
  m = U_\rr(t) \in \SD \quad (t \geq \tau).
\]
To find the parameter $\tau$, we will derive some conditions on the parameters
$(\sigma_1,\sigma_3,\ell)$. To show condition \eqref{eq:s-conds-m}, we proceed
for each fixed $j \in J$ as follows. Let $i, k \in J \setminus \set{j}$
be the other two indices complementary to $j$. We further write
\begin{equation} \label{eq:m-reduced}
  m' = g_i \cdot g_k \andq g = g_j,
\end{equation}
noting that
\[
  s_g( m' ) = s_g( m ).
\]

Our goal is to find an expression for $m'$ in terms of $g$.
In view of \eqref{eq:g-j-def} we can effectively rewrite $g_i$ and $g_k$ as
\[
  g_\nu = r_\nu \frac{g-1}{r_j} + 1 \quad (\nu = i,k).
\]
We then derive initially the expression
\begin{equation} \label{eq:m-g-1}
  m' = \frac{\sigma_3}{r_j} \, \mleft( \frac{g-1}{r_j} \mright)^{\!\!2}
    + ( \sigma_1 - r_j ) \frac{g-1}{r_j} + 1,
\end{equation}
where all terms and fractions still yield integers. Since we need an expansion
in $g$, we finally attain to the following expression for $m'$ with rational
coefficients.
\begin{equation} \label{eq:gamma-exp}
  m' = \gamma_0 + \gamma_1 \, g + \gamma_2 \, g^2
\end{equation}
with
\begin{equation} \label{eq:gamma-def}
  \gamma_0 = \beta + 1, \quad \gamma_1 = -(\alpha+\beta), \quad \gamma_2 = \alpha,
\end{equation}
obeying
\[
  \gamma_0 + \gamma_1 + \gamma_2 = 1
\]
where
\begin{equation} \label{eq:alpha-def}
  \alpha = \frac{\sigma_3}{r_j^3}, \quad
  \beta = \frac{\sigma_3}{r_j^3} - \frac{\sigma_1}{r_j} + 1.
\end{equation}

We deduce from Lemma~\ref{lem:sigma-alpha} that
\[
  \alpha, \beta, \gamma_\nu \in
  \begin{cases}
    \ZZ, & \text{if } r_j=1, \\
    \ZZ/r_j^2 \setminus \ZZ, & \text{otherwise}.
  \end{cases}
\]

The case $r_j = 1$ can only happen when $j=1$, while the coefficients are integers.
In the other case the coefficients are fractions. However, there arises the problem
of finding a suitable $g$-adic expansion of \eqref{eq:gamma-exp} to show that
in fact $s_g(m') = g$. To proceed in this way, we let \phrase{the coefficients
$\gamma_\nu$ float}. We have to distinguish between the following two cases.

\Case{$r_j = 1$}
We rewrite \eqref{eq:gamma-exp} by \eqref{eq:gamma-def} and \eqref{eq:alpha-def} as
\[
  m' = a_0 + a_1 \, g + a_2 \, g^2
\]
with the coefficients
\begin{align}
  a_0 &= \sigma_3 - \sigma_1 + 2, \nonumber \\
  a_1 &= \lambda \, g - (2\sigma_3 - \sigma_1 + 1), \label{eq:coeff-a1} \\
  a_2 &= \sigma_3 - \lambda, \nonumber
\end{align}
and the parameter $\lambda \in \set{1,2}$.

Next we show that the integers $a_\nu$ are $g$-adic digits, so satisfying
\begin{equation} \label{eq:g-digits}
  g > a_\nu \geq 0 \quad (\nu=0,1,2),
\end{equation}
which implies that
\[
  s_g(m') = a_0 + a_1 + a_2 =
  \begin{cases}
    g,    & \text{if $\lambda = 1$},\\
    2g-1, & \text{if $\lambda = 2$}.
  \end{cases}
\]

By Lemma~\ref{lem:sigma-equal} we have the inequalities
\[
  \sigma_3 \geq \sigma_1 \geq 6,
\]
and by \eqref{eq:g-j-def} that
\begin{equation} \label{eq:coeff-g}
  g = \sigma_3 t + \ell + 1.
\end{equation}

Thus, we infer that \eqref{eq:g-digits} holds for $a_0$ and $a_2$, if $t \geq 1$
and $\lambda = 1,2$. For $a_1$ we first consider \eqref{eq:coeff-a1} with
$\lambda = 1$. The inequalities
\[
  \underbrace{\sigma_3 t + \ell + 1}_{g} \geq
  \underbrace{2\sigma_3 - \sigma_1 + 1}_{g \, - \, a_1} > 0
\]
are valid for $t \geq 2$ unconditionally, and for $t = 1$ if
$\ell \geq \sigma_3 - \sigma_1$.
Hence, \eqref{eq:g-digits} holds for $a_1$ in these two cases.

We now consider the remaining case $t = 1$ and $\ell < \sigma_3 - \sigma_1$
with $\lambda = 2$. From~\eqref{eq:coeff-a1} and~\eqref{eq:coeff-g}
we then derive the inequalities
\[
  g = \sigma_3 + \ell + 1 > a_1 = 2\ell + \sigma_1 + 1 > 0,
\]
which are valid by assumption, showing that \eqref{eq:g-digits} also holds for
$a_1$ in that case.

Finally, we achieve the conditions for the case $r_j = 1$ as
\begin{align}
  \tau &= \begin{cases}
    2, & \text{if $\ell < \sigma_3 - \sigma_1$},\\
    1, & \text{otherwise},
  \end{cases} \nonumber\\
\shortintertext{as well as}
  s_g(m') &=
  \begin{cases}
    g,    & \text{if $t \geq \tau$},\\
    2g-1, & \text{if $(\tau,t) = (2,1)$}.
  \end{cases} \label{eq:s-tau-1}
\end{align}
This completes the first case $r_j = 1$.

\Case{$r_j > 1$}
We rewrite \eqref{eq:gamma-exp} as
\[
  m' = a_0 + a_1 \, g + a_2 \, g^2
\]
with the coefficients
\begin{align*}
  a_0 &= \gamma_0 + (1 - \fracpart{\gamma_2}) \, g, \\
  a_1 &= \gamma_1 + \fracpart{\gamma_2} \, g - (1 - \fracpart{\gamma_2}), \\
  a_2 &= \gamma_2 - \fracpart{\gamma_2}. \\
\shortintertext{By~\eqref{eq:gamma-def} and~\eqref{eq:alpha-def} these equations
turn into}
  a_0 &= (1 - \fracpart{\alpha}) \, g + \beta + 1, \\
  a_1 &= \fracpart{\alpha} \, g - (\alpha - \fracpart{\alpha}) - (\beta + 1), \\
  a_2 &= \alpha - \fracpart{\alpha}. \\
\shortintertext{Since $\theta = \fracpart{\alpha} \, g - \beta \in \ZZ$ by
Lemma~\ref{lem:sigma-theta}~(i) and $\intpart{\alpha} = \alpha - \fracpart{\alpha}$,
we finally arrive at the simplified equations}
  a_0 &= g - (\theta - 1), \\
  a_1 &= \theta - (1 + \intpart{\alpha}), \\
  a_2 &= \intpart{\alpha}.
\end{align*}
One observes that the coefficients $a_\nu$ $(\nu=0,1,2)$ are integers.
Moreover, they satisfy that
\[
  a_0 + a_1 + a_2 = g.
\]
There remains to show that the coefficients $a_\nu$ are in fact proper $g$-adic
digits, implying that $s_g(m') = g$ as desired.

For $a_2$ and $t \geq 1$ this easily follows from \eqref{eq:g-j-def} and
\eqref{eq:alpha-def} so that
\[
  g = r_j \, ( \sigma_3 \, t + \ell ) + 1
    > \intpart{\sigma_3/r_j^3} = \intpart{\alpha} = a_2 \geq 0.
\]
By Lemma~\ref{lem:sigma-theta}~(ii) and~\eqref{eq:g-inequal}, we have for
$t \geq 1$ the inequalities
\[
  g > \theta > 1 + \intpart{\alpha},
\]
which finally imply that $a_0, a_1 \in \set{1, \dots, g-1}$.
As a result, we conclude in the case $r_j > 1$ that
\begin{equation}\label{eq:s-tau-2}
  \tau = 1 \andq s_g(m) = g \quad (t \geq \tau).
\end{equation}

Now we consider the special case $j=3$, $t=0$, and $\rr \neq (1,2,3)$.
By Theorem~\ref{thm:Ur-prop} we have $U_{\rr}(t) > 1$, $\ell > 0$, and $g > 1$.
Since $r_1 < r_2 < r_3$, we infer that
\[
  \alpha = \sigma_3/r_j^3 < 1.
\]
Therefore $\alpha = \fracpart{\alpha}$ and $\intpart{\alpha} = 0$.
The coefficients $a_\nu$ then become
\[
  a_0 = g - (\theta - 1), \quad a_1 = \theta - 1, \quad a_2 = 0.
\]
We can apply Lemma~\ref{lem:sigma-theta}~(iii) and obtain
by~\eqref{eq:g-inequal2} the inequalities
\[
  g > \theta \geq 1.
\]

Comparing~\eqref{eq:eta-def} and~\eqref{eq:g-alpha-beta} yields
\[
  \theta = \alpha g - \beta
    = \frac{\sigma_1}{r_j} + \frac{\ell \sigma_3}{r_j^2} - 1 = \eta - 1,
\]
where $\eta \geq 2$ by Lemma~\ref{lem:sigma-eta}.
If $\theta > 1$ or equivalently $\eta > 2$, then
\[
  g > \theta > 1,
\]
implying that $a_0, a_1 \in \set{1, \dots, g-1}$ and $s_g(m') = g$. Otherwise,
we have the case $\theta = 1$ and $\eta = 2$. This yields $m' = g$ and thus
$s_g(m') = 1$. Consequently,
\begin{equation}\label{eq:s-g-cases}
  s_g(m') =
  \begin{cases}
    1, & \text{if $\eta = 2$}, \\
    g, & \text{if $\eta > 2$}.
  \end{cases}
\end{equation}
This completes the second case $r_j > 1$.
\smallskip

Combining both cases $r_j = 1$ and $r_j > 1$ yields that
\[
  \tau =
  \begin{cases}
    2, & \text{if $r_1 = 1$ and $\ell < \sigma_3 - \sigma_1$}, \\
    1, & \text{otherwise}.
  \end{cases}
\]
As a result, if $t \geq \tau$, then
\[
  m = U_\rr(t) = g_1 \cdot g_2 \cdot g_3 \in \SD.
\]
If $g_1$, $g_2$, and $g_3$ are odd primes, then $m \in \CP_3$ by
Theorem~\ref{thm:s-decomp}~(iii). This finishes the proof of the theorem.
\end{proof}

\begin{proof}[Proof of Theorem~\ref{thm:main2}]
We continue seamlessly with the proof of Theorem~\ref{thm:main} and
consider the complementary cases
\[
  m = U_\rr(t) = g_1 \cdot g_2 \cdot g_3 \quad (0 \leq t < \tau).
\]
We have to show three parts (in order of their dependencies).

(iii).~If $(\tau,t) = (2,1)$, then we obtain by \eqref{eq:s-tau-1} and
\eqref{eq:s-tau-2} that
\[
  s_{g_1}(m) = 2g_1 - 1, \quad s_{g_2}(m) = g_2, \quad s_{g_3}(m) = g_3.
\]
Thus, $m \in \SDG$ and its $s$-decomposition
$g_1 \cdot g_2 \cdot g_3 \in \SDG \setminus \SD$.

(i).~Assume that the factors $g_\nu$ are odd primes. Theorem~\ref{thm:Chernick2}
shows that $m \in \CN_3$. If $m \in \SD$, then $m \in \CP_3$ by
Theorem~\ref{thm:s-decomp}~(iii). But if $(\tau,t) = (2,1)$,
then part~(iii) implies that $m \notin \CP_3$.

(ii).~We consider the case $t=0$ and $j=3$. We then have the equality
$\vartheta = \eta$ by \eqref{eq:eta-def}. If $\rr=(1,2,3)$, then we obtain
$m=1$ by \eqref{eq:Ur-simple} and $\eta = 2$ by Lemma~\ref{lem:sigma-eta}.
Since $s_1(m) = 0$ by definition and $g_1 = g_2 = g_3 = 1$, the result follows.
If $\rr \neq (1,2,3)$, then the implications follow from~\eqref{eq:s-g-cases}.
For $\eta = 2$ we have by \eqref{eq:m-reduced} and \eqref{eq:s-g-cases}
that $m' = g_3 = g_1 g_2$, so $m = g_3^2$. If $\eta > 2$, then
$s_{g_3}(m) = g_3$ by~\eqref{eq:s-g-cases}, and Lemma~\ref{lem:estim-g-m}
implies that $m > g_3^2$. This completes the proof of the theorem.
\end{proof}

\begin{proof}[Proof of Theorem~\ref{thm:except}]
Let $m \in \CN_3 \setminus \CP_3$, where
\[
  m = p_1 \cdot p_2 \cdot p_3
\]
with odd primes $p_1 < p_2 < p_3$. Theorem~\ref{thm:CN-prop} implies that
$m \in \SDG$. From Theorem~\ref{thm:s-decomp}~(iii), it follows that
\[
  m \notin \CP_3 \impliesq m \notin \SD.
\]
By Theorem~\ref{thm:Chernick3} there exist unique $\rr \in \RS$ and
$t \geq 0$ such that
\[
  m = U_{\rr}(t),
\]
while Theorem~\ref{thm:main} implies that $0 \leq t < \tau$ with some
$\tau \in \set{1,2}$, since $m \notin \SD$. Next we consider two cases as follows.

\Case{$t = 0$} Since $m \in \SDG$ is no square, we infer from
Theorem~\ref{thm:main2}~(ii) that~\eqref{eq:cond-p} holds for $p_3$.

\Case{$t = 1$} From Theorem~\ref{thm:main2}~(iii), it follows
that~\eqref{eq:cond-p} holds for $p_2$ and $p_3$.

Hence, both cases imply that $m \in \SL$. This finally yields
$m \in (\SDG \cap \SL) \setminus \SD$, showing the result.
\end{proof}


\section{Proofs of Theorems \pref{thm:SD-prop}, \pref{thm:SL-prop},
\pref{thm:SD-inf}, \pref{thm:SD-bound}, and \pref{thm:CS-3}}
\label{sec:proofs-3}

The remaining proofs are given in this section, since they depend on
Theorems~\ref{thm:main} and~\ref{thm:except}. Recall the definitions of
Sections~\ref{sec:decomp} and~\ref{sec:except}. In the following we use the
notation $m = p_1 \dotsm p_n$, which means that $p_1 < \dotsb < p_n$ are
odd primes.

\begin{proof}[Proof of Theorem~\ref{thm:SD-prop}]
We have to show three parts.

(i).~Theorem~\ref{thm:CN-prop} implies that $\CN \subset \SDG$ by definition.

(ii).~First we show that $\CP \subseteq \SD \cap \CN$. If $m \in \CP \subset \CN$,
then $m$ is squarefree and $m = p_1 \dotsb p_n$ with $n \geq 3$, which is a
strict $s$-decomposition by definition of $\CP$. Thus, $m \in \SD \cap \CN$.
Next we show that $\CP \neq \SD \cap \CN$. We search for a counterexample by
constructing numbers lying in $\SD$. To do so, we consider as in
\eqref{eq:Ur-simple} again
\begin{equation}\label{eq:proof-Ur-1}
  U_{\rr}(t) = (6\,t+1)(12\,t+1)(18\,t+1) \quad (\rr = (1,2,3)).
\end{equation}
As a result of Theorem~\ref{thm:main}, we have that
\begin{equation}\label{eq:proof-Ur-2}
  U_{\rr}(t) \in \SD \quad (t \geq 1).
\end{equation}
We then find $m = U_{\rr}(5)$ with its strict $s$-decomposition and prime
factorization as
\[
  m = 172\,081 = 31 \cdot 61 \cdot 91 = 7 \cdot 13 \cdot 31 \cdot 61.
\]
One verifies by Korselt's criterion that $m \in \CN$. But since $s_7(m) = 19$,
$m$ fails to be in $\CP$. This finally implies that $\CP \subset \SD \cap \CN$.

(iii).~If $m \in \CP_3 \subset \CN_3$, then $m = p_1 \cdot p_2 \cdot p_3$ is
also a strict $s$-decomposition. Therefore, $m \in \SD \cap \CN_3$.
Contrary, if $m \in \SD \cap \CN_3$, then $m \in \CP_3$ by
Theorem~\ref{thm:s-decomp}~(iii). It follows that $\CP_3 = \SD \cap \CN_3$.
This finishes the proof of the theorem.
\end{proof}

\begin{proof}[Proof of Theorem~\ref{thm:SL-prop}]
We have to show two parts.

(i).~By definition we have $\CS \subseteq \CN \setminus \CP$.
We use the first example of $\CS_4$, namely,
\[
  m = 954\,732\,853 = 103 \cdot 109 \cdot 277 \cdot 307.
\]
We have $14$ proper divisors of $m$ (excluding $1$ and $m$). By construction of
$\CS$ we have $s_p(m) \neq p$ for each prime divisor $p \mid m$. A computational
check (e.g., with \textsc{Mathematica}) of the remaining ten proper divisors
$g \mid m$ shows each time that $s_g(m) \neq g$ , so $m \notin \SL$.
Finally, it follows that $\CN \setminus \CP \not\subset \SL \setminus \SD$.

(ii).~By Theorem~\ref{thm:except} we have
$\CN_3 \setminus \CP_3 \subseteq \SL \setminus \SD$. Considering the computed
examples with only two prime factors, we find that, for example,
$6 \in \SL \setminus \SD$, while $6 \notin \CN_3 \setminus \CP_3$.
It follows that $\CN_3 \setminus \CP_3 \subset \SL \setminus \SD$.

This completes the proof of the theorem.
\end{proof}

\begin{proof}[Proof of Theorem~\ref{thm:SD-inf}]
We have to show that $\SD$ is infinite. It suffices to use the example in
\eqref{eq:proof-Ur-1}. By Theorem~\ref{thm:main} and~\eqref{eq:proof-Ur-2},
this already implies that infinitely many values of $U_{\rr}(t)$, being
strictly increasing for $t \geq 1$, lie in $\SD$.
\end{proof}

\begin{proof}[Proof of Theorem~\ref{thm:SD-bound}]
Define the real-valued function and its inverse for $x,y \in \RR_{\geq 0}$ by
\[
  f(x) := \frac{1}{11} \, x^{1/3} - \frac{1}{3}, \quad
  f^{-1}(y) = 1331 \left( y + \frac{1}{3} \right)^{\!3}.
\]
We have to show that
\begin{equation} \label{eq:estim-SD-f}
  S'(x) > f(x) \quad (x \geq 1).
\end{equation}
While $f(x)$ is strictly increasing for $x \geq 0$, the function $S'(x)$
increases stepwise, counting elements of $\SD$ less than $x$.
Considering the first values of $\SD = \set{45,96,\dotsc}$, we have
\[
  S'(1) = 0, \quad S'(46) = 1, \andq S'(97) = 2.
\]
From
\[
  f(0) = -1/3, \quad f^{-1}(0) = 49.29\dotsc, \andq f^{-1}(1) = 3154.96\dotsc,
\]
we infer that \eqref{eq:estim-SD-f} holds for $x \in [1,3154]$.
By Theorem~\ref{thm:main} and relations \eqref{eq:proof-Ur-1}
and~\eqref{eq:proof-Ur-2} we have
\[
  g(t) := (6\,t+1)(12\,t+1)(18\,t+1) \withq g(t) \in \SD \quad (t \in \NN).
\]

Note that $f^{-1}(y) > g(y)$ for $y \geq 0$, as verified by
\[
  f^{-1}(y) - g(y) = 35 \, y^3 + 935 \, y^2 + \frac{1223}{3} \, y + \frac{1304}{27}.
\]

Since $S'(x)$ increases after each $x = g(t)$ for $t \in \NN$ and $S'(97) = 2$,
we conclude for $x > g(1) = 1729$ that
\begin{align*}
  S'(x) &> 1 + \#\set{t \in \NN : g(t) < x} \\
    &\geq 1 + \#\set{t \in \NN : f^{-1}(t) < x} \\
    &\geq f(x).
\end{align*}
Combining both intervals for $x$ shows \eqref{eq:estim-SD-f} and the result.
\end{proof}

\begin{proof}[Proof of Theorem~\ref{thm:CS-3}]
By definition we have $\CS_3 \subseteq \CN_3 \setminus \CP_3$.
Let $m = p_1 \cdot p_2 \cdot p_3 \in \CN_3 \setminus \CP_3$.
Theorem~\ref{thm:except} shows that $s_{p_3}(m) = p_3$, implying that
$m \notin \CS_3$. As a consequence, we infer that $\CS_3 = \emptyset$.
This proves the theorem.
\end{proof}

\vspace*{8pt}


\section{Taxicab Numbers}
\label{sec:taxicab}

As noted in \eqref{eq:taxicab}, the smallest number which can be written as
the sum of two positive cubes in two ways is the number $1729$, known as
Ramanujan's taxicab number or the Hardy--Ramanujan number.

By Section~\ref{sec:decomp} we have the relations
\[
  1729 = 7 \cdot 13 \cdot 19 \in \CP_3 \subset \SD_* \subset \SDG_*.
\]

The $n$th taxicab number $\Ta(n)$ is defined to be the smallest number which
can be written as the sum of two positive cubes in $n$ ways. The next numbers
$\Ta(n)$ for $n=3,4$ were listed by Silverman \cite{Silverman:1993}.
Subsequently, Wilson \cite{Wilson:1999} found $\Ta(5)$, while C.~S.~Calude,
E.~Calude, and Dinneen \cite{CCD:2003} and Hollerbach \cite{Hollerbach:2008}
announced $\Ta(6)$ (see also OEIS~\cite[Seq.\,A011541]{OEIS}).
Table~\ref{tbl:taxi1} reports these numbers.

\begin{table}[H]
\begin{center} \small
\begin{tabular}{r@{$\;=\;$}l}
  \toprule
  $87\,539\,319$ & $3^3 \cdot 7 \cdot 31 \cdot 67 \cdot 223$ \\
  $6\,963\,472\,309\,248$ & $2^{10} \cdot 3^3 \cdot 7 \cdot 13 \cdot 19
    \cdot 31 \cdot 37 \cdot 127$ \\
  $48\,988\,659\,276\,962\,496$ & $2^6 \cdot 3^3 \cdot 7^4 \cdot 13 \cdot 19
    \cdot 43 \cdot 73 \cdot 97 \cdot 157$ \\
  $24\,153\,319\,581\,254\,312\,065\,344$ & $2^6 \cdot 3^3 \cdot 7^4 \cdot 13
    \cdot 19 \cdot 43 \cdot 73 \cdot 79^3 \cdot 97 \cdot 157$ \\
  \bottomrule
\end{tabular}

\caption{Taxicab numbers $\Ta(n)$ for $n=3,\dots,6$.}
\label{tbl:taxi1}
\end{center}
\end{table}

Similarly, allowing only cube-free numbers, one finds in \cite{Silverman:1993}
and \cite[Seq.\,A080642]{OEIS} the corresponding taxicab numbers $\Tc(n)$ for
$n=3,4$, as listed in Table~\ref{tbl:taxi2}.

\begin{table}[H]
\begin{center} \small
\begin{tabular}{r@{$\;=\;$}l}
  \toprule
  $15\,170\,835\,645$ & $3^2 \cdot 5 \cdot 7 \cdot 31 \cdot 37 \cdot 199 \cdot 211$ \\
  $1\,801\,049\,058\,342\,701\,083$ & $7 \cdot 31 \cdot 37 \cdot 43 \cdot 163
    \cdot 193 \cdot 9151 \cdot 18\,121$ \\
  \bottomrule
\end{tabular}

\caption{Cube-free taxicab numbers $\Tc(n)$ for $n=3,4$.}
\label{tbl:taxi2}
\end{center}
\end{table}

A quick computational check reveals that all taxicab numbers of
Tables~\ref{tbl:taxi1} and~\ref{tbl:taxi2} have a common property that
\[
  \Ta(n), \Tc(m) \in \SDG_* \setminus \SD_* \quad (n=3,\dots,6, \; m=3,4).
\]

Therefore, one may raise the following question.

\begin{question}
Is there a link between the sets $\SDG_*$, $\SD_*$ and certain integral solutions
of the elliptic curve $X^3 + Y^3 = A$?
\end{question}


\section{Polygonal Numbers}
\label{sec:poly}

The polygonal numbers (cf.~\cite[pp.~38--42]{Conway&Guy:1996}) can be defined
as follows. For any integer $h \geq 1$, define an \emph{$h$-gonal number} by
\[
  \GN^h_n = \frac{1}{2} \mleft(n^2(h-2) - n(h-4)\mright) \quad (n \geq 1).
\]
Special cases are, e.g., the triangular numbers
\begin{align*}
  \TN_n &= \GN^3_n = \binom{n+1}{2} = \frac{1}{2}n(n+1) \\
\shortintertext{and the hexagonal numbers}
  \HN_n &= \GN^6_n = \binom{2n}{2} = n(2n-1),
\end{align*}
while $\GN^4_n = n^2$ are the squares, and $\GN^2_n = \GN^n_2 = n$ give the
trivial cases. For $h = 1$ there are only the special cases
$\GN^1_1 = \GN^1_2 = 1$; otherwise, $\GN^1_n \leq 0$ for $n \geq 3$.

Recall the definition of a universal form $U_{\rr}(t)$ in \eqref{eq:Ur-def},
as well as the definitions of $\sigma_\nu$ and $\ell$ in
\eqref{eq:sigma-1}~--~\eqref{eq:param-ell}.
We further use the definitions and results of Section~\ref{sec:proofs-2}.

The following theorem shows that for any given $\rr \in \RS$
all values of $U_{\rr}(t)$ for $t \geq 0$ are polygonal numbers.

\begin{theorem} \label{thm:poly}
Let $\rr \in \RS$ and
\[
  U_\rr(t) = g_1 \cdot g_2 \cdot g_3
\]
where
\[
  g_\nu = r_\nu \, ( \sigma_3 \, t + \ell ) + 1 \quad (\nu=1,2,3).
\]
Then we have for $t \geq 0$ and $\nu=1,2,3$ the relations
\[
  U_\rr(t) = \GN^{h_\nu}_{g_\nu} \withq h_\nu = 2( c_\nu + d_\nu \, t),
\]
where $c_\nu$ and $d_\nu$ are positive integers given by
\[
  c_\nu = \frac{\sigma_1}{r_\nu} + \frac{\ell \sigma_3}{r_\nu^2} \geq 2 \andq
  d_\nu = \mleft( \frac{\sigma_3}{r_\nu} \mright)^{\!\!2} \geq 4.
\]
In particular,
\[
  h_\nu \geq
  \begin{cases}
    4,  & \mbox{if $t=0$}, \\
    12, & \mbox{if $t \geq 1$}.
  \end{cases}
\]
\end{theorem}

\begin{proof}
Set $J = \set{1,2,3}$ and fix $j \in J$. Let $i, k \in J \setminus \set{j}$
with $i \neq k$. We solve for $h$ with $g = g_j$ the equation
\begin{equation}\label{eq:poly-m}
  \GN^{2h}_{g} = g \cdot g_i \cdot g_k.
\end{equation}
After some simplifications the equation turns into
\[
  (g-1)(h-1) = g_i \cdot g_k - 1.
\]
From \eqref{eq:m-reduced} and \eqref{eq:m-g-1}, we derive that
\[
  h - 1 = \frac{\sigma_3}{r_j^3} (g-1) + \frac{\sigma_1}{r_j} - 1
  = \frac{\sigma_3}{r_j^2} (\sigma_3 t + \ell) + \frac{\sigma_1}{r_j} - 1.
\]
Thus,
\[
  h = \frac{\sigma_1}{r_j} + \frac{\ell \sigma_3}{r_j^2}
  + \left( \frac{\sigma_3}{r_j} \right)^{\!\!2} t = c_j + d_j \, t.
\]
Lemma~\ref{lem:sigma-eta} shows that $c_j \geq 2$ is a positive integer.
Since $r_j \mid \sigma_3$ and $\sigma_3 \geq 6$ by Lemma~\ref{lem:sigma-equal},
we infer that $\sigma_3 / r_j \geq 2$ and so $d_j \geq 4$.
With $h_j = 2h$ and $g_j = g$, the result follows from~\eqref{eq:poly-m}.
In particular, we then obtain for $t=0$ and $t \geq 1$ the estimates
$h_j \geq 4$ and $h_j \geq 12$, respectively.
This completes the proof of the theorem.
\end{proof}

\begin{corollary}
All $3$-factor Carmichael numbers are polygonal numbers. More precisely,
if $m \in \CN_3$, then for each prime divisor $p$ of $m$ there exists a
computable integer $h \geq 6$ such that
\[
  m = \GN^h_p.
\]
\end{corollary}

\begin{proof}
Let $m \in \CN_3$. By Theorem~\ref{thm:Chernick3} there exist $\rr \in \RS$ and
$t \geq 0$ such that $m = p_1 \cdot p_2 \cdot p_3 = U_\rr(t)$.
Fix $j \in \set{1,2,3}$ and set $p = p_j$. Applying Theorem~\ref{thm:poly} yields
$m = \GN^{h}_{p}$ with a computable even integer $h \geq 4$. Since $\GN^4_p = p^2$,
the case $h = 4$ cannot occur, so we finally infer that $h \geq 6$.
\end{proof}

We can go further into this connection between polygonal numbers, universal
forms, and Carmichael numbers. Considering the factors $g_\nu$ of a number
$m$ instead of its parametric representation $m = U_{\rr}(t)$ leads to a more
general result. The following identity explains this elementary relationship
in the context of Korselt's criterion.

\begin{theorem} \label{thm:poly2}
We have the identity
\begin{equation} \label{eq:poly-ident}
  m = \GN^h_g \withq h = 2 \left( \frac{m/g-1}{g-1}+1 \right).
\end{equation}
For $g, m \in \NN$ and $g \neq 1$, the identity holds if $h \geq 1$ is
integral. There are the following statements:
\begin{enumerate}
\item The trivial cases are
\begin{align*}
  m = g \geq 2 &\iffq h = 2, \\
  m \geq 1, \; g=2 &\iffq h = m \geq 1.
\end{align*}

\item If $m$ is a Carmichael number and $g$ is a prime divisor of $m$,
      then identity~\eqref{eq:poly-ident} holds where $h \geq 6$ is even.

\item For $n \geq 3$ let $U_n(t) = g_1 \dotsm g_n$ be a universal form as
      defined in~\eqref{eq:Un-def}, where $g_\nu = a_\nu \, t + b_\nu$
      $(1 \leq \nu \leq n)$. For fixed $\nu$ and $t \geq 0$, let $m = U_n(t)$
      where $m > g = g_\nu > 1$. Then identity~\eqref{eq:poly-ident} holds
      where $h \geq 4$ is even.
\end{enumerate}
\end{theorem}

\begin{proof}
It is easy to verify that the expression $\GN^h_g$ in \eqref{eq:poly-ident}
simplifies to~$m$. Let $g, m \in \NN$ where $g \neq 1$. Since
\[
  d := \frac{m/g-1}{g-1} > -1,
\]
it follows that $h > 0$. If $h$ is integral, then $h \geq 1$ and
\eqref{eq:poly-ident} holds. We have to show three parts.

(i).~Let $g > 1$. We infer that
\[
  m=g \iffq d = 0 \iffq h = 2,
\]
showing the first equivalence.
Let $m \geq 1$. If $g = 2$, then $h = m$. Conversely, $h=m$ implies
the equation $m = 2((m/g-1)/(g-1)+1)$ with solution $g = 2$.
This shows the second equivalence.

(ii).~Let $m \in \CN$ and $g \mid m$ be a prime divisor.
From Korselt's criterion it follows that
\begin{equation} \label{eq:congr-m-g}
  m - 1 \equiv \frac{m}{g} - 1 \pmod{g-1}.
\end{equation}
Since $m > g > 1$, it follows that $d \in \NN$. The case $d = 1$ would imply
$m = g^2$, contradicting that $m$ is squarefree. Finally, this implies that
$h \geq 6$ is integral and even, showing that \eqref{eq:poly-ident} holds.

(iii).~By \eqref{eq:Un-prop} a universal form $U_n(t)$ for $n \geq 3$ satisfies
\[
  U_n(t) \equiv 1 \pmod{g_\nu - 1},
\]
whenever $g_\nu > 1$. For fixed $t \geq 0$, $m = U_n(t)$, and $g = g_\nu > 1$,
congruence \eqref{eq:congr-m-g} follows from $g \mid m$. As $m > g$, we infer that
$h \geq 4$ is integral and even, implying that \eqref{eq:poly-ident} holds.
This completes the proof of the theorem.
\end{proof}

The following example demonstrates the interplay of the preceding results.

\begin{example}
Interestingly, the parameter
\[
  \alpha = \sqrt{\frac{66\,337}{181 \cdot 733}}
    = 1 \Big/ \sqrt{2 - \frac{1}{66\,337}} = 0.7071\dotsc
\]
in Theorem~\ref{thm:CP-prop} (note that $132\,673 = 181 \cdot 733$) depends on
the number
\[
  m = 8\,801\,128\,801 = 181 \cdot 733 \cdot 66\,337 = \HN_{66\,337} \in \CP,
\]
which is the least hexagonal number $\HN_p$ in $\CP$ (see \cite{Kellner&Sondow:2021}).
Since $m \in \CP_3$, Theorem~\ref{thm:poly} furthermore implies that
\[
  m = U_\rr(0) = \GN^h_p
\]
for some $\rr \in \RS$. Indeed, by Theorem~\ref{thm:Chernick3} one finds
$\rr = (15, 61, 5528)$, $\sigma_1 = 5604$, $\sigma_3 = 5\,058\,120$, and
$\ell = 12$. A computation verifies that
\[
  p = r_3 \ell + 1 = 66\,337, \quad h
    = 2 \left( \frac{\sigma_1}{r_3} + \frac{\ell \sigma_3}{r_3^2} \right) = 6,
\]
while Theorem~\ref{thm:poly2} shows in another way that
\begin{align*}
  h &= 2 \left( \frac{181 \cdot 733 - 1}{66\,337 - 1} + 1 \right) = 6. \\
\shortintertext{A third formula follows from a $p$-adic approach by
\cite[Cor.~4.3]{Kellner&Sondow:2021} that}
  h &= 2 \left( \intpart{\frac{181 \cdot 733}{66\,337}} + 2 \right) = 6.
\end{align*}
\end{example}

As a final application of Theorem~\ref{thm:poly}, we obtain the following result
for the taxicab number $1729$.

\begin{example}
Let $\rr = (1,2,3) \in \RS$. We have $\sigma_1 = \sigma_3 = 6$ and $\ell = 0$ by
Table~\ref{tbl:Ur-1}. Theorem~\ref{thm:poly} provides the relations
\[
  U_\rr(t) = \GN^h_g \quad (t \geq 0)
\]
for
\[
  g = 6 \nu \, t + 1, \quad
  h = 2 \left( \frac{6}{\nu} + \left( \frac{6}{\nu} \right)^{\!\!2} t \right)
  \quad (\nu=1,2,3).
\]
Since $U_\rr(1) = 1729$, we obtain the unified formula
\[
  1729 = \GN^h_p
\]
for
\[
  p = 6\nu+1, \quad h = 4 \TN_{6/\nu}
    = 2 \left( \frac{6}{\nu} + \left( \frac{6}{\nu} \right)^{\!\!2} \right)
  \quad (\nu=1,2,3),
\]
which yields at once the known relations
\[
  1729 = \GN^{84}_{7} = \GN^{24}_{13} = \GN^{12}_{19}.
\]
\end{example}

\vspace*{8pt}


\end{document}